\documentclass[final]{amsart} 
\usepackage{amsmath}
\usepackage{amsfonts} \usepackage{amsthm} \usepackage{amssymb}
\usepackage{mathrsfs} \DeclareMathAlphabet{\mathpzc}{OT1}{pzc}{m}{it}
\usepackage{array}
\usepackage{longtable}
\usepackage{graphicx}
\usepackage{hyperref}
\hypersetup{colorlinks=true,linkcolor=black,anchorcolor=black,citecolor=black,filecolor=black,menucolor=black,runcolor=black,urlcolor=black,final}
\DeclareGraphicsExtensions{.png,.jpeg,.pdf}
\usepackage{tikz}
\usepackage{thmtools}
\usepackage{mathtools} \usepackage{color} \usepackage{enumitem}
\usetikzlibrary{decorations.markings,arrows,calc,intersections}
\usepackage{ragged2e}
\usepackage{showlabels}
\usepackage{amsthm}
\declaretheorem[numberwithin=section]{theorem}
\usepackage{subcaption}
\usepackage{enumitem}
\usepackage{stackengine,scalerel}
\newcommand\ttimes{\mathbin{\ThisStyle{\ensurestackMath{%
  \stackengine{-1\LMpt}{\SavedStyle\times}
  {\SavedStyle_{\hstretch{.9}{\mkern1mu\sim}}}{O}{c}{F}{T}{S}}}}}
\newcommand\pmttimes{\mathbin{\ThisStyle{\ensurestackMath{%
  \stackengine{-1\LMpt}{\SavedStyle\times}
  {\SavedStyle_{\hstretch{.9}{\mkern1mu(\sim)}}}{O}{c}{F}{T}{S}}}}}
\newcounter{claimcounter}

\def\centerarc[#1](#2)(#3:#4:#5)
    { \draw[#1] ($(#2)+({#5*cos(#3)},{#5*sin(#3)})$) arc (#3:#4:#5); }

 \newcommand{\inj}{\xhookrightarrow{}}
 \newcommand{\ZZ}{\mathbb{Z}}
 
\newcommand{\RR}{\mathbb{R}} \newcommand{\QQ}{\mathbb{Q}}
\newcommand{\HH}{\mathbb{H}}

\newcommand{\sub}{\subseteq}

\newtheorem{lemma}[theorem]{Lemma}
\newtheorem{prop}[theorem]{Proposition}
\newtheorem{corr}[theorem]{Corollary}
\theoremstyle{definition}
\newtheorem{defn}[theorem]{Definition}

\newtheorem{convention}[theorem]{Convention}
  \newcommand{\norm}[1]{\left\lVert#1\right\rVert}
     \newcommand{\del}{\partial}
\def\centerarc[#1](#2)(#3:#4:#5)
    { \draw[#1] ($(#2)+({#5*cos(#3)},{#5*sin(#3)})$) arc (#3:#4:#5); }

\title[Minimal triangulation size of Seifert fibered spaces]{Minimal triangulation size of Seifert fibered spaces with boundary}

\author{Adele Jackson}
\address{University of Oxford}
\email{\href{mailto:adele.jackson@maths.ox.ac.uk}{adele.jackson@maths.ox.ac.uk}}
\date{\today}

\def\subjclassname{\textup{2020} Mathematics Subject Classification}
\expandafter\let\csname subjclassname@1991\endcsname=\subjclassname
\expandafter\let\csname subjclassname@2000\endcsname=\subjclassname
\thanks{The author was supported by the Clarendon Fund, the Oxford-Australia Trust, and Lincoln College.}

\def\subjclassname{\textup{2020} Mathematics Subject Classification}
\expandafter\let\csname subjclassname@1991\endcsname=\subjclassname
\expandafter\let\csname subjclassname@2000\endcsname=\subjclassname

\begin{document}

\begin{abstract}
    One measure of the complexity of a 3-manifold is its triangulation complexity: the minimal number of tetrahedra in a triangulation of it. A natural question is whether we can relate this quantity to its topology. We determine the triangulation complexity of Seifert fibered spaces with non-empty boundary in terms of their Seifert data, up to a multiplicative constant. We also show that all singular fibres of such a Seifert fibered space (aside from those of multiplicity two) can be made simplicial in the $79^{th}$ barycentric subdivision of any triangulation of it.
\end{abstract}

\maketitle
\section{Introduction}

The \emph{triangulation complexity} $\Delta(M)$ of a compact 3-manifold $M$ is the minimum number of tetrahedra in any triangulation of $M$. 
This invariant, while easy to define, is challenging to determine.
A few specific natural families are known exactly, for example:
\begin{enumerate}
    \item if $\Sigma$ is a closed orientable surface of genus $g$, $\Delta(\Sigma\times I) = 10g-4$~\cite{BucherFrigerioPagliantini,JJSTBoundsGenusNormalSurface};
    \item if $M$ is a hyperbolic once-punctured torus bundle, its monodromy ideal triangulation (which is the same as its veering triangulation in the sense of Agol) is minimal~\cite{JRSTCuspedHyperbolicFamily};
    \item a minimal triangulation of the lens spaces $L(2n, 1)$ and $L(4n, 2n-1)$ is achieved by their minimal layered triangulations~\cite{JRTInfiniteLensSpaceFamily,JRTCoverings}.
\end{enumerate}
The triangulation complexity is known up to multiplicative bounds for closed hyperbolic mapping tori~\cite{LacPur} and for closed 3-manifolds with elliptic or sol geometries~\cite{LacPurSol}.
It is known precisely for a few other exact families~\cite{JRSTZ2Norm2,JRSTDehnFilling,RSTCuspedHyperbolicFamilyII,VesninFominykh}, but unfortunately the most notable characteristic of most of these families is solely that we know their triangulation complexity; there have been relatively few results about families that naturally arise in other contexts.

We use the definition of \emph{triangulation} that is conventional among low-dimensional topologists: that it is a collection of tetrahedra, with gluing maps between pairs of faces, which when executed produce a manifold homeomorphic to $M$. We take the convention in our proof that this triangulation is material (as opposed to the ideal case, in which one removes open regular neighbourhoods of the vertices from the triangulation to obtain boundary components), though as we will see in Lemma~\ref{lemma:idealdoesntmatter}, our lower bound on $\Delta(M)$ also holds for the ideal case.

Triangulation complexity arises as a means of enumerating 3-manifolds.
Censuses of 3-manifolds are constructed by considering all possible face pairings of up to $n$ tetrahedra for some $n$.
Compare this to enumerating knots by crossing number: this is another natural invariant whose connection to the geometry and topology of the underlying space (in this case, the knot complement) can be difficult to understand.
The 3-manifold censuses have been advanced to thirteen tetrahedra in the closed orientable irreducible case~\cite{MatveevTarkaev}, and in the cusped hyperbolic case, up to ten tetrahedra with possible repeats~\cite{VesninTarkaevFominykh} and up to eight with no repeats~\cite{BurtonSnapPeaComplete}.
For these two classes, one additional motivation is that triangulation complexity and Matveev complexity agree, so long as the Matveev complexity is not zero~\cite[Theorem 5]{Matveev90}.

Like Matveev complexity, triangulation complexity is submultiplicative over finite covers.
Unlike Matveev complexity and Gromov's simplicial volume, it is is not additive under connect sum: its bound in this setting is $\Delta(M\#N) \leq \Delta(M) + \Delta(N) + 4$~\cite[Construction 4.1]{Barchechat}.
This connect sum bound is in fact sharp\footnote{
The author used Regina~\cite{regina} to compute $\Delta(M\# N)$ for all $M$ and $N$ of triangulation complexity at most two, by constructing a census of closed orientable possibly reducible 3-manifolds of complexity at most 7. One example showing the bound is sharp is $\Delta(L(4,1)\#L(4,1)) = 6$ and $\Delta(L(4,1)) = 1$.
}.
Advantageously, however, only a finite number of manifolds have a given triangulation complexity.

Bounding triangulation complexity from above (at least up to a constant) is usually more straightforward than from below, as one can simply construct an efficient triangulation.
One tool for bounding it from below is that it is at least the simplicial volume (for a proof, see Proposition 0.1 in~\cite{FrancavigliaFrigerioMartelli}).
As, when $M$ is hyperbolic, its simplicial volume is proportional to its hyperbolic volume~\cite[Proposition 6.1.4]{Thurston}, this quantity can be interpreted geometrically.
However, when the simplicial volume of $M$ is zero -- that is, when $M$ is Seifert fibered or a graph manifold -- this bound is of no help to us.
This raises the question of whether we can get good estimates for $\Delta(M)$ in these cases.
We will show that, when $M$ is Seifert fibered with non-empty boundary, a particular natural triangulation of $M$ is within a multiplicative constant of the minimal triangulation size. The precise statement is as follows.

\begin{restatable}{theorem}{complexitybound}
There exists $k>0$ such that for any Seifert fibered manifold $M$ with non-empty boundary other than the solid torus, whose Seifert data is $[\Sigma, (p_1, q_1),\ldots,(p_n, q_n)]$ with (without loss of generality) $0 < q_i < p_i$ for each $i$,
$$\frac{1}{k}\left(|\chi(\Sigma)| + \sum_{i=1}^n \norm{q_i/p_i} + 1\right) \leq \Delta(M) \leq k\left(|\chi(\Sigma)| + \sum_{i=1}^n \norm{q_i/p_i} + 1 \right).$$
 \label{thm:complexitybound}
\end{restatable}

The solid torus, with $\Delta(D^2\times S^1) = 1$, is an exception as it has infinitely many non-isomorphic Seifert fibrations (it is $[D^2, (p, q)]$ for all $(p, q)$), and so no such lower bound can hold for it.

The term $\norm{q/p}$ is defined as follows.
The positive continued fraction of a reduced fraction $q/p$ with $q,p > 0$ is a sequence of integers $[a_0,\ldots,a_n]$ with $a_i > 0$ for all $i>0$ such that 
\begin{equation*}
    \frac{q}{p} = a_0 + \frac{1}{a_1 + \frac{1}{a_2 + \frac{1}{\cdots + \frac{1}{a_n}}}}.
\end{equation*}
The requirement that $a_i > 0$ for $i >0$ means that this continued fraction is unique.
Then $\norm{\frac{q}{p}}$ is defined to be $\sum_{i=0}^n a_i$.

This triangulation complexity result is very different to the Matveev complexity situation: there, it is known that the Matveev complexity of $M$ is at most $\sum_{i=1}^n \max(\norm{p_i/q_i}-3, 0)$~\cite[Theorem 1.2]{FominykhWiest}, and, in particular, is zero if $M$ has no singular fibres.

\begin{figure}[ht]
  \centering
  \resizebox{!}{.10\textheight}{\includegraphics{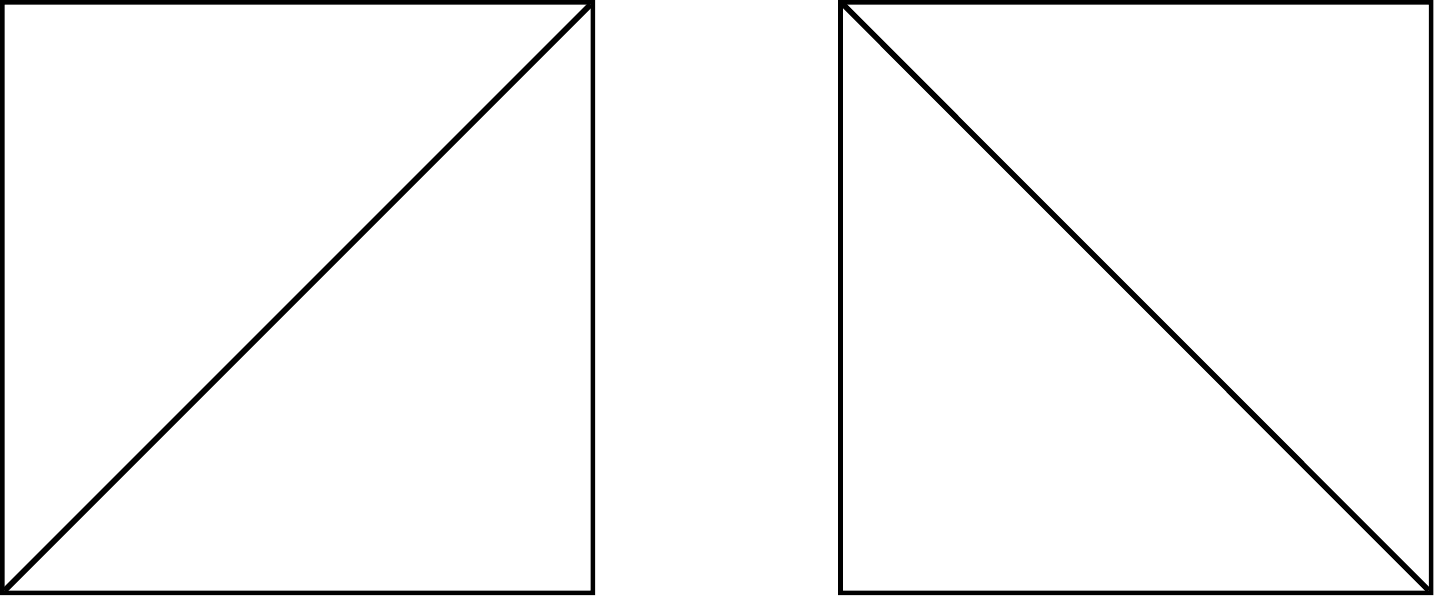}}
  \caption{A 2-2 Pachner move.}
  \label{fig:22_pachner}
  \vspace*{-15ex}  
	\begin{center}
	$\to$
	\end{center}
	\vspace*{11ex}
\end{figure}

\begin{figure}[ht]
  \centering
  \begin{subfigure}[b]{\textwidth}
  	\centering
   	\resizebox{0.45\textwidth}{!}{\includegraphics{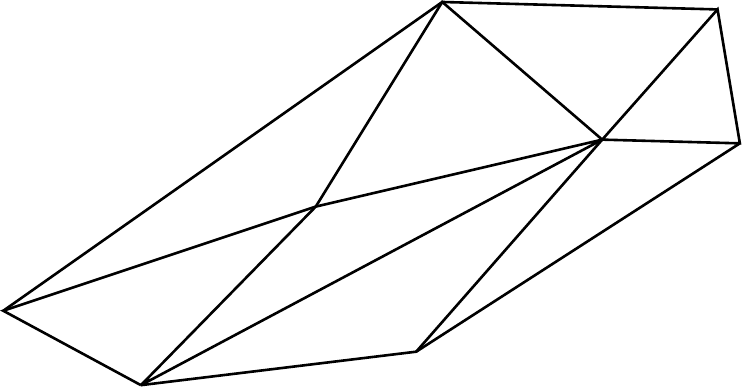}}
   	\resizebox{0.45\textwidth}{!}{\includegraphics{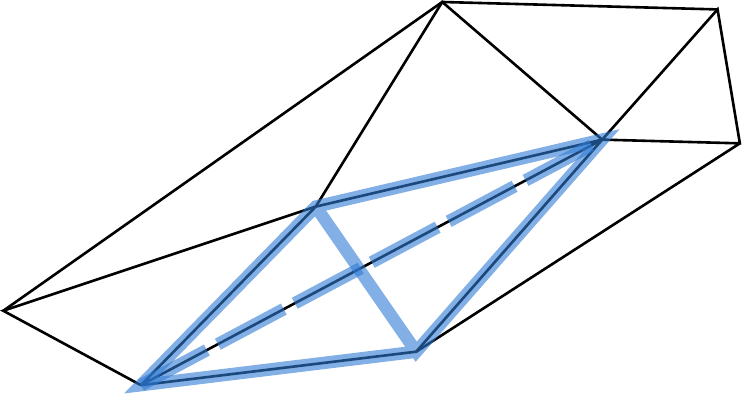}}
	\caption{Gluing the blue tetrahedron to the surface.}
  \end{subfigure}
 \begin{subfigure}[b]{\textwidth}
  	\centering
   	\resizebox{0.45\textwidth}{!}{\includegraphics{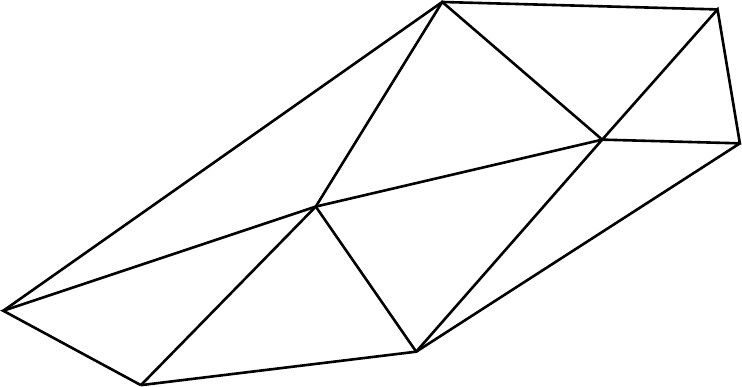}}
	\caption{The new induced surface triangulation.}
  \end{subfigure}  \caption{Given a surface with an induced triangulation, if we glue on the blue tetrahedron, the new induced triangulation is the result of a 2-2 Pachner move.}
  \label{fig:layered_triang}
\end{figure}

The triangulation that gives the upper bound, which we describe more precisely in the proof of Proposition~\ref{prop:upper_bound}, is roughly as follows: let $\Sigma'$ be $\Sigma$ with $n$ additional boundary components added.
Triangulate $\Sigma'\times S^1$ using $k|\chi(\Sigma')|$ tetrahedra.
For each singular fibre, glue a \emph{layered solid torus} with an edge of slope $\frac{q_i}{p_i}$ into one of the boundary components.
These are triangulations of the solid torus built by taking its triangulation by a single tetrahedron~\cite{Baker} and gluing tetrahedra to it by identifying two faces of a tetrahedron with two adjacent faces of the boundary triangulation.
The effect of such a gluing is to execute a \emph{2-2 Pachner move} on the boundary triangulation.
A 2-2 Pachner move is shown in Figure~\ref{fig:22_pachner}, and Figure~\ref{fig:layered_triang} depicts an example of performing one by gluing on a tetrahedron.

For the lower bound, as we show in Lemma~\ref{lemma:easylowerbound} the linear dependence on $|\chi(\Sigma)|$ is straightforward; it is the singular fibres that pose some difficulty.
The main technical result we will use to handle them is of independent interest.
\begin{restatable}{theorem}{simplicialfibres}
    \label{thm:simplicialfibres}
    Let $M$ be a Seifert fibered space with non-empty boundary and let $\mathcal{T}$ be a (material) triangulation of $M$.
    The collection of singular fibres of $M$ that are not of multiplicity two have disjoint simplicial representatives in $\mathcal{T}^{(79)}$, the $79^{th}$ barycentric subdivision of $\mathcal{T}$.
    Furthermore, in $\mathcal{T}^{(82)}$, these simplicial singular fibres have disjoint simplicial solid torus neighbourhoods such that there is a simplicial meridian curve of length 48 for each such neighbourhood.
\end{restatable}

In Section~\ref{section:continuedfractions} we discuss the connection between continued fractions, the Farey graph, and layered solid tori. We give background on handle structures and parallelity bundles in Section~\ref{section:parallelity}. When cutting a manifold along a normal surface, parallelity bundles are a means of efficiently organising and studying the pieces that lie between parallel normal discs.
In Section~\ref{section:simplicialfibres}, the heart of the paper, we prove Theorem~\ref{thm:simplicialfibres} by analysing the parallelity bundle, and then in Section~\ref{section:triangulationbounds} we prove the main theorem.

\section{Continued fractions and the Farey graph}
\label{section:continuedfractions}

The term $\norm{q/p}$ in Theorem~\ref{thm:complexitybound} arises via Dehn filling tori to create singular fibres with parameter $q/p$.
This operation glues in a solid torus by the map taking its meridian to an edge of slope $q/p$ in the (section, fibre) coordinates on a surface bundle $M$, which we write as $(\mu, \lambda)$. 

\begin{prop}
Let $T$ be a one-vertex triangulation of the torus, and let $\lambda$ and $\mu$ be two of its edges.
Let $U$ be the manifold resulting from gluing a solid torus to $T$ by the map taking its meridian to $p\mu + q\lambda$, where $0 < |q| < p$.
There is a triangulation of $U$, with induced boundary triangulation $T$, by at most $\norm{\frac{|q|}{p}}+2$ tetrahedra.
\label{prop:layeredtriangulationcount}
\end{prop}

Consider the Farey graph in the hyperbolic plane: identifying $\HH^2$ with the upper half plane, this is the graph whose vertices are the points of $\QQ\cup \{\infty\}$ on the real line, and whose edges are, for each pair of reduced fractions $p/q$ and $r/s$ such that $|ps-qr| = 1$ (interpreting $\infty$ as $\frac{1}{0}$), the geodesic in $\HH^2$ between the corresponding vertices.
The Farey tessellation is the resulting tessellation of $\HH^2$ into ideal triangles.

Let $Tr(T^2)$ be the one-vertex triangulation graph of the torus (up to isotopy). Its vertices are one-vertex triangulations of the torus and its edges are 2-2 Pachner moves (as in Figure~\ref{fig:22_pachner}).
There is a natural identification of $Tr(T^2)$ and the Farey tessellation.
The number $|ps-qr|$ is exactly the intersection number between curves of slopes $p/q$ and $r/s$ on the torus.
Interpret a triangle of the Farey tessellation, whose vertices are three slopes that pairwise intersect once, as the triangulation of $T^2$ by those slopes.
A 2-2 Pachner move consists of removing one of the slopes and replacing it with the unique other possible slope to complete the triangulation, so corresponds to crossing an edge of the tessellation.

When we glue in a layered solid torus, we start with a single tetrahedron triangulation of the solid torus that induces a one-vertex triangulation of its boundary (see~\cite{Baker} for an illustration).
Let $k$ and $l$ be two of its edges, chosen so that the coordinate $m = l-k$ is a meridian and $l$ is a longitude.
Note then that the three edges of this triangulation of $T^2$ are $l$, $l-m$ and $2l-m$.
As illustrated in Figure~\ref{fig:layered_triang}, gluing on a tetrahedron by two of its faces has the effect on the boundary triangulation of a 2-2 Pachner move.
We can thus glue on one tetrahedron so that the boundary triangulation is $(l, l-m, m)$.

The gluing map we wish to execute is $m\mapsto p\mu + q\lambda$ and $l\mapsto r\mu + s\lambda$, where $r$ and $s$ are integers so that $ps-qr = 1$.
It is important to note that as our only requirement was the meridian gluing, any such $r$ and $s$ will do.

It is thus sufficient to determine how many tetrahedra we need to glue on to our original triangulation $(\mu, \lambda, \lambda \pm \mu)$ such that the new boundary triangulation has an edge of slope $q/p$.
This is equivalent to asking for the distance in the Farey tessellation from $(0, \infty, \pm 1)$ to the line of vertices that are adjacent to the $q/p$ slope.
Note that, up to the choice of $\pm 1$, this distance is the same for $q/p$ and $-q/p$, so this quantity is equal to the distance from the line of vertices adjacent to the $|q|/p$ slope to the closer of $(0, \infty, 1)$ and $(0, \infty, -1)$.
As we have required that $0 \leq |q| \leq p$, this is equal to the distance from the \emph{line} of faces adjacent to $\infty$ to the line of faces adjacent to $|q|/p$, which is $\norm{|q|/p}-1$~\cite[Lemma 5.4]{LacPurSol}.
We may need to add an extra tetrahedron to take us from $(0, \infty, 1)$ to $(0, \infty, -1)$ or vice versa.
Thus this layered triangulation possibly first has a tetrahedron to achieve the closest choice of the $\pm 1$ parameter, and then has $\norm{q/p}-1$ tetrahedra to change the coordinates to have a $q/p$ slope edge, and finally two tetrahedra to have the effect of gluing in a solid torus with the correct meridian slope: a total of $\norm{|q|/p}+2$ tetrahedra.
We can note that the configuration $(0, \infty, \pm 1)$ is invariant under exchanging the coordinates.
Doing this takes $q/p$ to $p/q$, and we do indeed see that $\norm{|q|/p} = \norm{p/|q|}$.

We will need the following lemma which gives another symmetry in the continued fraction sum.

\begin{lemma}
\label{lemma:negative_same_continued_sum}
Suppose $0 < q < p$.
Then $\norm{\frac{q}{p}} = \norm{\frac{p-q}{p}}$.
\end{lemma}

\begin{proof}
Note that $\norm{\frac{q}{p}}-1$ is the distance from the line about $\frac{q}{p}$ to the line about infinity.
By reflecting in the vertical line through $1/2$, which takes $\frac{q}{p}$ to $\frac{p-q}{p}$, we see that $\norm{\frac{q}{p}}-1$ is equal to the distance from the line about $\frac{p-q}{p}$ to infinity, which is $\norm{\frac{p-q}{p}}-1$.
\end{proof}

\section{Handle structures and parallelity bundles}
\label{section:parallelity}
We will mostly work in the setting of handle structures rather than triangulations, as they are more amenable to being cut along normal surfaces.
\begin{convention}
We take the definition of a handle structure to require the following:
\begin{enumerate}
    \item each $k$-handle, with product structure $D^k\times D^{3-k}$, intersects the handles of lower index in exactly $\del D^k \times D^{3-k}$, and is disjoint from the other $k$-handles;
    \item 1-handles and 2-handles intersect in a manner compatible with their respective product structures; that is, a 1-handle $D^1\times D^2$ intersects a 2-handle $D^2\times D^1$ in segments of the form $D^1\times \gamma$ in the 1-handle and $\lambda\times D^1$ in the 2-handle, where $\gamma$ and $\lambda$ are collections of arcs in $\del D^2$ in the respective product structures.
\end{enumerate}
\label{convention:handlestructure}
\end{convention}

As all our 3-manifolds have non-empty boundary, by the following lemma we can assume that they have no 3-handles.
\begin{lemma}
Let $\mathcal{H}$ be a handle structure of a connected 3-manifold with non-empty boundary.
There is a collection of 2-handles in $\mathcal{H}$ such that removing these 2-handles and all the 3-handles in $\mathcal{H}$ does not change the resulting 3-manifold up to homeomorphism.
\label{lemma:remove3cells}
\end{lemma}

\begin{proof}
Form a cell complex for the 3-manifold from $\mathcal{H}$ that has a $(3-k)$-cell for each $k$-handle, glued in the corresponding way. Also glue in a $(3-k-1)$-cell for each component of intersection of a $k$-handle and the boundary.
The interior cells of this complex are dual to handles of $\mathcal{H}$.
As this cell complex is connected, its 1-skeleton is connected, so each vertex has a path of edges from it to the boundary.
The last edge in this path is dual to a 2-handle which is incident to a 3-handle.
Remove this pair of handles.
By induction, we can repeat this to remove all the 3-handles.
\end{proof}

Although we are working with material triangulations of $M$ rather than ideal ones, the following well-known result demonstrates that our lower bound also holds in the ideal setting.
An ideal triangulation is a collection of tetrahedra with face pairing maps such that the space formed from gluing together the tetrahedra and then removing all their vertices is homeomorphic to the interior of $M-\del M$.

\begin{lemma}
Let $\Delta(M)$ be the minimal number of tetrahedra in a material triangulation of a compact 3-manifold $M$ with boundary, and let $\Delta'(M)$ be the minimal number of tetrahedra in an ideal triangulation of $M$.
Then $\Delta(M) \leq 14\Delta'(M)$.
\label{lemma:idealdoesntmatter}
\end{lemma}

\begin{proof}
Consider an ideal triangulation of $M$.
Truncate the tetrahedra to obtain a cell structure homeomorphic to $M$.
Cone each non-triangle face, each of which has six sides, to a vertex of the face.
Note that each of these six-sided faces is now triangulated by four triangles, so the 3-cells now have 20 faces each.
In each truncated tetrahedron, pick a vertex of highest valence: it will be of valence at least six, and so neighbours at least six triangles.
Cone the tetrahedron to this vertex, giving a triangulation of this cell with one tetrahedron for each triangle that did not neighbour the vertex, the count of which is at most 14.
Thus we have realised a material triangulation of $M$ using $14\Delta'(M)$ tetrahedra.
\end{proof}

\subsection{Parallelity bundles}
\label{subsection:parallelitybundles}

The notion of a parallelity bundle was introduced by Lackenby~\cite{LacComposite} to tackle the following situation: we often know that some interesting surface in a 3-manifold $M$ has a fundamental normal representative $F$.
We might wish to cut along this surface, but naively constructing a triangulation of the complement will give us exponentially many new tetrahedra.
However, most of these tetrahedra are not interesting.
They lie in the regions between parallel normal discs and so naturally are $I$-bundles: these regions constitute the \emph{parallelity bundle}.
This notion can be expressed more naturally in the language of handle structures.

\begin{defn}
Let $\mathcal{T}$ be a triangulation (or cell structure) of a 3-manifold $M$.
The \emph{dual handle structure} $\mathcal{H}$ for $M$ is formed by taking one $(3-k)$-handle for each $k$-simplex (or cell) of $\mathcal{T}$ that is not contained in the boundary and gluing them in the corresponding way.
\label{defn:dual_handle}
\end{defn}

\begin{defn}
A 1-manifold properly embedded in a surface $S$ equipped with a handle structure is \emph{standard} if it is disjoint from the 2-handles, intersects the 0-handles in arcs, and intersects each 1-handle $D^1\times D^2$ in a collection of arcs, each of which is $D^1\times\{*\}$ for some point in $D^2$.
\end{defn}

\begin{defn}
Suppose that $S$ is a subsurface of the boundary of a 3-manifold $M$.
We say that a handle structure $\mathcal{H}$ of $M$ is a \emph{handle structure for the pair $(M, S)$} if $\del S$ is standard with respect to the induced handle structure on $\del M$.
\end{defn}

This definition is motivated by the following situation: suppose that $S$ is a normal surface with respect to a handle structure for $M$.
(See Matveev for the relevant definitions~\cite[\S3.4]{Matveev}.)
The induced handle structure on $M\backslash\backslash S$ is a handle structure for $(M\backslash\backslash S, S)$.

\begin{defn}
Suppose that $\mathcal{H}$ is a handle structure for $(M, S)$.
A handle $H$ of $\mathcal{H}$ is a \emph{parallelity handle} with respect to $S$ if we can endow it with a product structure $D^2\times D^1$ such that $H\cap S = D^2\times \del D^1$, and so that if $H'$ is another handle that intersects $H$, then $H\cap H'$ has the form $\gamma\times D^1$, where $\gamma$ is a subset of $\del D^2$.

The \emph{parallelity bundle} of $(M, S)$ is the union of the parallelity handles with respect to $S$.
\end{defn}

At this point, to understand $M\backslash\backslash S$, there are two approaches we could take.
First, we could use the algorithm of Agol, Hass and Thurston~\cite{AgolHassThurston} to efficiently determine the topological type of the parallelity bundle: that is to compute, for each component, which surface it is a bundle over, and how these components intersect the non-parallelity pieces.
This approach gives an algorithm that runs in polynomial time in the number of tetrahedra in $M$ and in the logarithm of the weight of $S$ (for details see Proposition 13 of~\cite{HarawayHoffman}), so it is efficient when $S$ is a fundamental normal surface.
However, we lose track of the inclusion of $M\backslash \backslash S$ into $M$ when we do this.
Alternatively, we can hope to make the pieces of the parallelity bundle topologically simple, and analyse them directly.
To do this, we expand our study to \emph{generalised} parallelity bundles: we extend the parallelity bundle over certain regions of $M\backslash\backslash S$ that are compatible with its fibered structure.

\begin{defn}
If $N\cong \Sigma\pmttimes I$ where $\Sigma$ is a surface, the \emph{vertical boundary} $\del_v N$ of $N$ is $\del\Sigma\pmttimes I$, and the \emph{horizontal boundary} $\del_h N$ of $N$ is $\Sigma\pmttimes \del I$.
\end{defn}

\begin{defn}[Definition 5.2~\cite{LacComposite}]
A \emph{generalised parallelity bundle} $\mathcal{B}$ in a handle structure $\mathcal{H}$ for $(M, S)$ is a subset of the handles of $\mathcal{H}$ such that:
\begin{enumerate}
	\item $\mathcal{B}$ is endowed with an $I$-bundle structure over a surface;
	\item any handle of $\mathcal{B}$ that intersects $\del_v \mathcal{B}$ is a parallelity handle, where the $I$-bundle structure on the parallelity handle agrees with that of $\mathcal{B}$;
	\item $\del_h \mathcal{B}$ is $B\cap S$;
	\item if $H$ is a handle in $\mathcal{B}$, and $H'$ is adjacent to $H$ and of higher index, then $H'$ is in $\mathcal{B}$ also.
\end{enumerate}
\end{defn}

As we would hope, the parallelity bundle is a generalised parallelity bundle~\cite[Lemma 5.3]{LacComposite}.
A \emph{maximal} generalised parallelity bundle is a generalised parallelity bundle that is not contained in any other generalised parallelity bundles.
We will see that, aside from the following situation, the parallelity bundle components have incompressible horizontal boundary in $S$.

\begin{defn}
Suppose that there are two annuli $G$ and $G'$ in $M$, with $\del G = \del G' \subset S$, such that $G$ is properly embedded and $G'$ is contained in the boundary.
Suppose that $G\cup G'$ bounds a union of handles $P$ in $M$ such that:
\begin{enumerate}
	\item either $P$ is a parallelity region with respect to $G$ and $G'$ or $P$ is contained in a 3-ball;
	\item if $H$ is a handle in $P$ and $H'$ is a handle of higher index that is adjacent to $H$, then $H'$ is also in $P$;
	\item if $H$ is a parallelity handle of $(M, S)$ intersecting $P$, then $H$ is in $P$;
	\item $G$ is a vertical boundary component of a generalised parallelity bundle lying in $P$;
	\item $G'\cap (\del M-S)$ is either empty or a regular neighbourhood of a core curve of $G'$.
\end{enumerate}
Removing the interiors of $P$ and $G'$ is an \emph{annular simplification}.
\end{defn}

An annular simplification of $(M, S)$ gives a smaller handle structure for the same 3-manifold, and the boundary of this handle structure inherits a copy of $S$.
See \S5 of~\cite{LacComposite} for further details.

Since in an annular simplification, $G$ is a vertical boundary component of a generalised parallelity bundle component, and any handle in such a component that intersects the vertical boundary must be a parallelity handle, we have the following result.

\begin{lemma}
If $\mathcal{H}$ does not contain any parallelity handles, it does not admit any annular simplifications.
\end{lemma}

\begin{theorem}[Proposition 5.6~\cite{LacComposite}]
Suppose $M$ is a compact orientable irreducible 3-manifold with a handle structure $\mathcal{H}$.
Let $S$ be an incompressible surface in $\del M$ that is not a 2-sphere such that $\del S$ is standard.
Suppose that $M$ does not admit any annular simplifications.
Let $\mathcal{B}$ be a maximal generalised parallelity bundle for $(M, S)$.
Then each component of $\mathcal{B}$ has incompressible horizontal boundary.
\label{thm:incompparallelitybundle}
\end{theorem}

\section{Constructing simplicial singular fibres}
\label{section:simplicialfibres}

In this section we will prove Theorem~\ref{thm:simplicialfibres}, which we now recall.

\simplicialfibres*

A related result was proved by Mijatovi\'{c}~\cite[Theorem 5.1]{MijatovicTriangulationsSFS}: he showed that a neighbourhood of each singular fibre (including those of multiplicity two) is simplicial in a suitable subdivision of $\mathcal{T}$.
His bound was that this subdivision could be obtained by making at most $2^{2^{400|\mathcal{T}|^2}}$ Pachner moves.
If our bound were phrased in this language, the number of Pachner moves would be linear in the size of the triangulation.

The result is trivial if $M$ has no singular fibres other than those of multiplicity two.
We will handle the case when $M$ is a solid torus separately.
So assume for the remainder of this section that $M$ is a Seifert fibered space that has $n$ singular fibres that are not of multiplicity two for some $n\geq 1$, and is not a solid torus.
Let $\mathcal{T}$ be a triangulation of $M$, and let $\mathcal{H}$ be the dual handle structure (as in Definition~\ref{defn:dual_handle}).

Our approach is to take a collection of normal annuli that cut out a neighbourhood of each singular fibre.
We will then prove that, for each of these singular fibres, in the induced handle structure of their solid torus neighbourhood there is a core curve consisting of a bounded number of arcs in each 0- and 1-handle that avoids the parallelity bundle altogether.

We will do this by applying a theorem that gives us a controlled core curve of the solid torus with respect to some affine data.
We will construct a handle structure $\mathcal{H}$ by taking the dual handle structure to a triangulation, cutting along normal surfaces, removing handles, and replacing parallelity bundle components with 2-handles.
Its 0-handles, which arise as pieces of tetrahedra cut along normal discs, each have a natural identification with a convex polyhedron in $\RR^3$ where their intersections with the 2- and 3-handles of $\mathcal{H}$ are unions of faces.
Viewing the 1-handles as thickened pieces of faces in the triangulation, these too have natural identifications with the product of a polygon $P$ and the interval, such that they are glued to 0-handles along $P\times \del I$ by an affine gluing map.
This is enough to upgrade $\mathcal{H}$ to an \emph{affine handle structure}, as defined in \S4 of~\cite{LacCoreCurves}, and allow us to apply the following result.

\begin{theorem}[Theorem 4.2~\cite{LacCoreCurves}]
   Let $\mathcal{H}$ be an affine handle structure of the solid torus. Suppose that each 0-handle has at most four components of intersection with  the 1-handles, and that each 1-handle has at most three components of intersection with the 2-handles. Then $M$ has a core curve that intersects only the 0- and 1-handles, that respects the product structure on the 1-handles, that intersects each 1-handle in at most 24 straight arcs, and that intersects each 0-handle in at most 48 arcs.
   Moreover, the arcs in each 0-handle are simultaneously parallel to a collection of arcs in the boundary of the corresponding polyhedron, and each component of this collection intersects each face of the polyhedron in at most six straight arcs.
   \label{thm:laccorecurves}
\end{theorem}

This notion of ``straight'' is induced by the affine structure on the face of the polyhedron.
Up to affine isotopy, there are only finitely many configurations of $n$ straight arcs in a given polygon.

\begin{defn}
A \emph{subnormal handle structure} $\mathcal{H}$ in a given handle structure $\mathcal{K}$ is a handle structure with an embedding $\mathcal{H} \inj \mathcal{K}$ such that the image of each handle $H$ of $\mathcal{H}$ is either a handle of $\mathcal{K}$, or is a subset of a handle $K$ of $\mathcal{K}$ that is a component of the complement of some collection of normal discs in $K$.

Let $\gamma$ be a curve in a handle structure $\mathcal{K}$.
A \emph{subnormal neighbourhood} of $\gamma$ is a subnormal handle structure $\mathcal{H}$ in $\mathcal{K}$ containing $\gamma$ that is homeomorphic to a solid torus for which $\gamma$ is a core curve.
\end{defn}

Let $S$ and $S'$ be 3-dimensional submanifolds of some 3-manifold $M$, where $S$ and $S'$ may have some boundary decoration.
We say that $S$ and $S'$ are \emph{isotopic in $M$} to mean that there is a (not necessarily $\del M$-fixing) isotopy, preserving the boundary decoration, of the embedding map of $S$ in $M$ to that of $S'$.

Let $A$ be a properly-embedded vertical annulus in $M$ that, together with an annulus $A'$ in the boundary with $\del A = \del A'$, separates a vertical neighbourhood of a singular fibre from $M$.
Take a disjoint family of one such $A_i$ for each non-multiplicity-two singular fibre $\gamma_i$.
As these annuli are essential since $M$ is not a solid torus, we can isotope them to be simultaneously normal.
Write $\{S_i\}$ for the collection of subnormal neighbourhoods of singular fibres that they separate from $M$, and write $\{\mathcal{H}_i\}$ for the induced handle structures of these subnormal neighbourhoods.
As annular simplifications do not change $(S_i, \del A_i)$ up to isotopy in $M$, and reduce the number of handles, perform them until no more are possible.

\subsection{Removing the parallelity bundle}

Fix our attention on one singular fibre neighbourhood $S_i$, recalling that the singular fibre which is its core curve is not of multiplicity two.
Let $\mathcal{B}$ be a maximal generalised parallelity bundle for $\mathcal{H}_i$ with respect to $A_i$.
As $A_i$ is essential and normal in $M$, so is incompressible in $\del S_i$ and has standard boundary, by Theorem~\ref{thm:incompparallelitybundle} the horizontal boundary of $\mathcal{B}$ is an incompressible subsurface of $\del S_i$ and is contained in $A_i$ by definition.
Thus each component of $\mathcal{B}$ is an $I$-bundle over a disc, annulus, or M\"obius band; in the annulus and M\"obius band bundle cases, their horizontal boundaries form a collection of annuli in $A_i$ which have core curves that are isotopic to the core curve of $A_i$.
We will prove that M\"obius band bundles do not occur since the singular fibre $\gamma_i$ is not of multiplicity two, and then will remove annulus bundles and any disc bundles that intersect $\del S_i$ in their vertical boundary without changing $(S_i, A_i)$ up to isotopy in $M$.

\begin{lemma}
Let $B$ be an annulus bundle or M\"obius band bundle component of $\mathcal{B}$. Each component of the boundary of $\del_v B$ intersects at most one boundary component of $A_i$.
\label{lemma:verticalboundaryinessential}
\end{lemma}

\begin{proof}
Note that, as $\del_h B$ is a collection of incompressible annuli and is a subset of $A_i$, the core curve of each component of $\del_h B$ is isotopic to the core curve of $A_i$.
If a component of the boundary of $\del_v B$ intersects both boundary components, the arc that intersects both components cuts $A_i$ into a disc, so intersects the core curve of some component of $\del_h B$, which is a contradiction.
\end{proof}

We will use this lemma to show that, if we remove such a parallelity bundle component and we know that this does not change $S_i$ up to isotopy in $M$, then it preserves the pair $(S_i, A_i)$ up to isotopy in $M$.

\begin{lemma}
If $B$ is a component of the generalised parallelity bundle that is a M\"obius bundle or annulus bundle, and some component of $\del_v B$ intersects the boundary of $S_i$ in its interior but is not a subset of $\del S_i$, then we can remove handles from the handle structure of $(S_i, A_i)$ such that this component is a subset of $\del S_i$ without altering the resulting submanifold up to isotopy in $M$.
\label{lemma:intersects_bdry_means_embedded}
\end{lemma}

\begin{proof}
If a component $\alpha$ of $\del_v B$ is not properly embedded but is also not contained in $\del S_i$, we can decompose this annulus as follows.
Since $\del_h B$ is incompressible and not a disc, $\del_v B$ is incompressible also so as $\alpha$ is in a solid torus it is boundary-parallel.
Consider the components of its intersection with each handle. These are discs with four sides: two that are its intersection with $A_i$, and the other two that may be in the interior of the solid torus or may be in $\del S_i - A_i$.
Colour each of these squares by whether its interior is in $\del S_i$ or not.
Then consider the tiling of $\alpha$ where we merge adjacent squares of the same colour (noting that, as $\alpha$ is neither properly embedded nor contained in the boundary, the whole annulus is not of only one colour).
Let $D$ be one of these tiles of $\del_v B$ whose interior is in the interior of the solid torus.
Note that $D$ is properly embedded: two of its edges are contained in $A_i$ and two are in $\del S_i - A_i$.
This disc cannot be a meridian disc for $S_i$ as then the singular fibre would have multiplicity two, so it must be boundary-parallel -- that is, $\del D$ bounds a disc in $\del S_i$.
We thus have three possibilities for how $\del D$ can lie with respect to $(S_i, A_i)$ up to homeomorphism, shown in Figure~\ref{fig:parallelitydiscinboundary}(a)-(c).
As the two arcs of $\del D$ in $A_i$ are parts of $\del_v B$ and so do not span $A_i$, they each start and end on one component of $\del A_i$.
They may intersect the same component or two different components, and in the case when they intersect the same component, they may be nested or not.
Up to homeomorphism, the two arcs in $\del S_i-A_i$ are determined by the fact that $\del D$ bounds a disc in $\del S_i$.

Consider the 3-ball $R$ that $D$ bounds.
Note that this 3-ball intersects $B$ only in $D$: if they intersected elsewhere, then since $B$ is connected and $D$ is separating in $M$, the entirety of $B$ would be contained in a 3-ball, but then its horizontal boundary would be compressible.
We can thus remove this 3-ball without affecting $(S_i, A_i)$ up to isotopy in $M$ by replacing $(S_i, A_i)$ with $(S_i-R, A_i-R)$ in the case of Figure~\ref{fig:parallelitydiscinboundary}(a) and (b), or with $(S_i-R, (A_i-R) \cup D)$ in the case of Figure~\ref{fig:parallelitydiscinboundary}(c).
We have reduced number of squares in the tiling of $\alpha$, so by induction we can repeat this process until $\alpha$ is a subset of $\del S_i$.
\end{proof}

Note that the cell structure of $(S_i, A_i)$ after this process may not satisfy Convention~\ref{convention:handlestructure} and so may not be a handle structure.
We will correct this when we remove the parallelity bundle component $B$.

\begin{figure}[th]
  \centering
  \begin{subfigure}[b]{0.40\textwidth}
  	\centering
   	\resizebox{0.7\textwidth}{!}{
        \includegraphics[clip, trim=0cm 7cm 7cm 0cm]{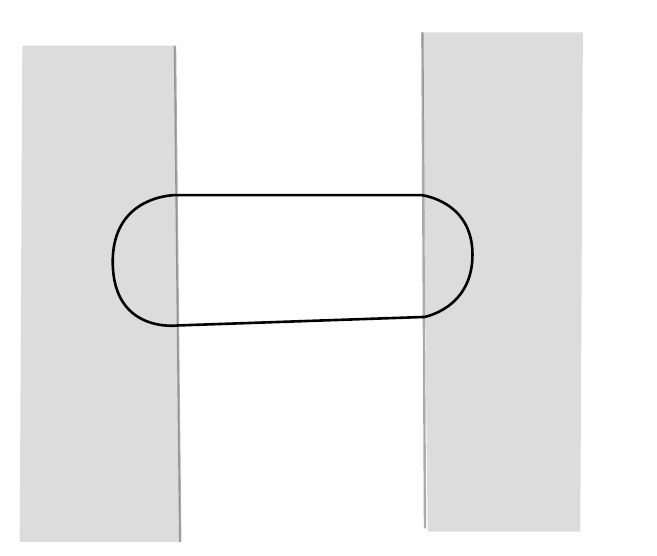}
    }
    \caption{}
  \end{subfigure}
  \begin{subfigure}[b]{0.40\textwidth}
  	\centering
   	\resizebox{0.7\textwidth}{!}{
        \includegraphics[clip, trim=7cm 7cm 0cm 0cm]{figures/discsintorus.pdf}
    }
    \caption{}
  \end{subfigure}
  \begin{subfigure}[b]{0.40\textwidth}
  	\centering
   	\resizebox{0.7\textwidth}{!}{
        \includegraphics[clip, trim=0cm 0cm 7cm 7cm]{figures/discsintorus.pdf}
    }
    \caption{}
  \end{subfigure}
  \begin{subfigure}[b]{0.40\textwidth}
  	\centering
   	\resizebox{0.7\textwidth}{!}{
        \includegraphics[clip, trim=7cm 0cm 0cm 7cm]{figures/discsintorus.pdf}
    }
    \caption{}
  \end{subfigure}
  \caption{The possible ways the boundary of a properly embedded disc in $\del_v B$ could intersect $(\del S_i, A_i)$. These diagrams show subsets of $\del S_i$. The shaded section is $A_i$. The black curve is $\del D$.}
  \label{fig:parallelitydiscinboundary}
\end{figure}

\begin{lemma}
The maximal generalised parallelity bundle has no M\"obius band bundle components.
\label{lemma:nomobiuscomponents}
\end{lemma}

\begin{proof}
Let $B$ be such a component, so $B$ has a bundle structure as the twisted product of a M\"obius band and an interval.
Consider the incompressible and hence boundary-parallel annulus that constitutes its vertical boundary, $\del_v B$.

If $\del_v B$ intersects $\del S_i$, then by Lemma~\ref{lemma:intersects_bdry_means_embedded} we can assume it is a subset of $\del S_i$.
Then, as its interior is disjoint from $A_i$ and its boundary components are subsets of $A_i$, it must be the entirety of $\del S_i - A_i$.
Now, as the boundary of $B$ is the entirety of $\del S_i$, we see that $S_i$ is the neighbourhood of a singular fibre of multiplicity 2, so we have a contradiction.

If $\del_v B$ {is} properly embedded, we can isotope $\del_v B$ to be vertical with respect to the Seifert fibration of $S_i$ that is induced from $M$.
Let $D_p$ be the orbifold that is a disc with one cone point of angle $\frac{2\pi}{p}$, recalling that $p$ is the multiplicity of this singular fibre.
The Seifert fibration of $M$ gives us an orbifold fibration $S_i \to D_p$ with $S^1$-fibres.
Now $\del_v B$ is the preimage of a properly-embedded arc under this map and $A_i$ is the preimage of a boundary arc.
The arc from $\del_v B$ has its endpoints on the arc from $A_i$, and the union of the arc from $\del_v B$ and the connected segment of the arc from $A_i$ cut out by these endpoints is the boundary of a sub-disc of $D_p$.
If this sub-disc does not contain the cone point, then the 3-manifold covering it is a solid torus whose meridian disc is bounded by one arc from $A_i$ and one from $\del_v B$: this is not a M\"obius band bundle.
If this sub-disc \emph{does} contain the cone point, then a meridian disc of the solid torus covering it is a $2p$-gon whose boundary alternates arcs from $\del_v B$ and $A_i$. To be a M\"obius band bundle, $p$ must be 2, which is a contradiction.
\end{proof}

\begin{lemma}
Let $\mathcal{H}$ be a handle structure for $(S, A)$ where $S$ is a 3-manifold and $A$ is a subsurface of its boundary.
Let $B$ be a component of a maximal generalised parallelity bundle with respect to $A$.
Then $\mathcal{H}-B$ satisfies Convention~\ref{convention:handlestructure} and so is a handle structure. 
\label{lemma:removeparallelityhandlestructure}
\end{lemma}

\begin{proof}
As the intersections of handles in $\mathcal{H}-B$ are a subset of those in $\mathcal{H}$, it is enough to show that each $k$ handle $H$ (with product structure $D^k\times D^{3-k}$) intersects the handles of lower index in exactly $\del D^k\times D^{3-k}$.
Certainly $H$ intersects the handles of lower index in a subset of this set.
If a handle adjacent to $H$ in $\mathcal{H}$ was in $B$, then this had to have been of higher index than $H$, as otherwise $H$ would also have been in $B$.
Thus we did not remove any intersections between $H$ and lower index handles.
\end{proof}

\begin{lemma}
There is a subnormal neighbourhood $S_i'$ of $\gamma_i$ with a distinguished annulus $A_i'$ in $\del S_i'$, where the induced handle structure for $S_i'$ from $\mathcal{H}$ is a handle structure for the pair $(S_i', A_i')$, such that $(S_i', A_i')$ and $(S_i, A_i)$ are isotopic in $M$, whose maximal generalised parallelity bundle has no annulus bundle or M\"obius band bundle components.
\label{lemma:removeannulusbundles}
\end{lemma}

\begin{proof}
We know that any subnormal handle structure in $M$ which is isotopic to $(S_i, A_i)$ in $M$ has no M\"obius band bundle components by Lemma~\ref{lemma:nomobiuscomponents}.
Suppose that $B$ is some annulus bundle component of the parallelity bundle of $(S_i, A_i)$.
If the interior of $\del_v B$ intersects $\del S_i$, then by Lemma~\ref{lemma:intersects_bdry_means_embedded} we can remove handles from $(S_i, A_i)$ without changing the resulting submanifold up to isotopy in $M$.
By induction, then, we can assume that if the interior of $\del_v B$ intersects $\del S_i$, then one of the components $\alpha$ of $\del_v B$ is contained in $\del S_i$.
As in the M\"obius bundle case, this component is the entirety of $\del S_i - A_i$.
The other component $\beta$ of $\del_v B$ must be properly embedded, as it is disjoint from $\alpha$ and its interior is disjoint from $A_i$.
We thus have one properly embedded component of $\del_v B$, and the other is either properly embedded or a subset of the boundary.

Again as in the M\"obius bundle case, we can isotope both $A_i$ and each component of $\del_v B$ to be vertical in the Seifert fibration; that is, so that they are preimages of arcs in $D_p$ under the projection $S_i\to D_p$.
The two components of $\del_v B$ project to parallel arcs.
The annulus $A_i$ projects to an arc in the boundary of $D_p$.

If a component of $\del_v B$ is contained in the boundary, cut along the other component.
We have two resulting pieces: the parallelity bundle, which we discard, and a piece $S_i'$ where $(S_i', A_i\cap S_i')$ is isotopic to $(S_i, A_i)$ in $M$.

Otherwise, suppose both components are properly embedded.
In $D_p$, one of the two parallel arcs from $\del_v B$ bounds a disc together with a segment of $A_i$ that does not contain the other arc of $\del_v B$.
This disc may or may not contain the cone point.
If it does contain the cone point, the singular fibre does not run through the parallelity bundle.
Cutting along the innermost arc (with respect to the cone point) has the same effect as in the previous case, except now one piece of the complement is the parallelity bundle glued to a solid torus along an annulus.
In particular, the other piece, $S_i'$, has the same core curve (when included into $M$) as $S_i$, and $(S_i', A_i\cap S_i')$ is isotopic to $(S_i, A_i)$ in $M$.
If the disc does not contain the cone point, we cut along the innermost arc with respect to the cone point, removing the parallelity bundle glued to a solid torus, and once again can see that the other piece, $S_i'$, has the same core curve as $S_i$, and $(S_i', (A_i\cup \del_v B)\cap S_i')$ is isotopic to $(S_i, A_i)$ in $M$, so we set the distinguished annulus in $S_i'$ to be $(A_i \cup \del_v B)\cap S_i'$.

As the number of handles is reduced by this operation, we can recompute the parallelity bundle and repeat it until there are no annulus bundle components of $\mathcal{B}$ remaining.
Note that after performing these modifications we have a handle structure by Lemma~\ref{lemma:removeparallelityhandlestructure}.
\end{proof}

All remaining components of the generalised parallelity bundle $\mathcal{B}$ in $(S_i, A_i)$ are disc bundles.

\begin{lemma}
If $B$ is a disc bundle component of the maximal generalised parallelity bundle such that its vertical boundary intersects $\del S_i$ in its interior, then there is a properly embedded disc, which is a subset of $\del_v B$, such that cutting along this disc separates a 3-ball from $S_i$ that contains $B$ and does not change $(S_i, A_i)$ up to isotopy in $M$. 
\label{lemma:remove_disc_bundles}
\end{lemma}

\begin{proof}
Suppose $\del_v B$ intersects $\del S_i$ in its interior.
Note that the annulus $\del_v B$ cannot be a subsurface of $\del S_i$: if it were, then as it intersects $A_i$ in its boundary only, it would have to have the same boundary curves as $A_i$, but boundary curves of $A_i$ do not bound discs (i.e. candidates for $\del_h B$) in $A_i$.
Thus, as in the proof of Lemma~\ref{lemma:intersects_bdry_means_embedded}, we can divide $\del_v B$ into squares according to whether a region of its interior is in $\del S_i$ or in the interior of $S_i$.
Again as in the previous case, since $p\not=2$, each of these squares is boundary-parallel.

Unlike in the previous case, the boundaries of the properly embedded squares may be nested in $\del S_i$.
Consider a square $D$ corresponding to an outermost such curve in $\del S_i$.
Suppose that $B$ is not contained in the ball $D$ bounds.
In this case, $B$ intersects this ball only in the square $D$, and as $\del D$ was an outermost curve, $B$ intersects the ball bounded by each properly-embedded disc in $\del_v B$ in only the properly-embedded disc itself.
We can thus see that $\del_v B$ is isotopic in $S_i$ to an annulus in $\del S_i$ that is disjoint from the rest of $B$, and in particular from $\del_h B$ (though possibly not from $A_i$).
But then both boundary curves of this annulus bound discs in $\del S_i$, which is impossible.

Cutting along $D$ removes a 3-ball $R$ from $S_i$ which contains the parallelity bundle component $B$.
There are four possible ways, up to homeomorphism, for $\del D$ to lie in $(\del S_i, A_i)$, which are the four configurations shown in Figure~\ref{fig:parallelitydiscinboundary}.
Replace $(S_i, A_i)$ with $(S_i-R, A_i-R)$ in the case of Figure~\ref{fig:parallelitydiscinboundary}(a) or (b), and with $(S_i-R, (A_i-R) \cup D)$ in the case of Figure~\ref{fig:parallelitydiscinboundary}(c) or (d).
This does not change $(S_i, A_i)$ up to isotopy in $M$.
\end{proof}

There may be parallelity bundle components remaining that we cannot remove without changing $(S_i, A_i)$ up to isotopy in $M$.
We will replace them with 2-handles to obtain a handle structure for $(S_i, A_i)$ where the 0- and 1-handles and the attaching annuli of the 2-handles naturally include into $M$ in a particularly nice way: this handle structure is \emph{$(0,0)$-nicely embedded subnormal}, which is a special case of being $(k,l)$-nicely embedded in the sense of Lackenby and Schleimer~\cite{LacSchleimerElliptic}.

\begin{defn}
Let $N$ be a manifold (piecewise linearly) embedded in a 3-manifold $M$, where $\mathcal{G}$ and $\mathcal{H}$ are handle structures for $N$ and $M$ respectively.
Let $\mathcal{G}'$ be the handle structure of just the 0- and 1-handles of $N$, and let $A\sub \del\mathcal{G}'$ be the attaching annuli of the 2-handles.
Then $\mathcal{G}$ is \emph{$(0,0)$-nicely embedded subnormal} in $\mathcal{H}$ if the following conditions hold:
\begin{itemize}
	\item each $k$-handle of $\mathcal{G}'$ is a subnormal subset of a $k$-handle in $\mathcal{H}$,
	\item for each 1-handle of $\mathcal{H}$, there is a product structure such that the product structures of all of the 1-handles of $\mathcal{G}'$ that lie in it are compatible with this product structure,
	\item in any $k$-handle of $\mathcal{H}$, none of the $k$-handles of $\mathcal{G}'$ lie between parallel normal discs in $\mathcal{H}$,
	\item the intersection of $A$ and any handle $G$ of $\mathcal{G}'$ is a union of components of intersection between $G$ and handles of $\mathcal{H}$.
\end{itemize}
\label{defn:nicelyembedded}
\end{defn}

\begin{corr}
By removing handles of $\mathcal{H}_i$ and replacing the maximal parallelity bundle with 2-handles, we can obtain a handle structure $\mathcal{G}$ for the pair $(S_i', A_i')$ that is $(0,0)$-nicely embedded subnormal in $\mathcal{H}$ and is isotopic to $(S_i, A_i)$ in $M$.
In particular, it has no parallelity 0- or 1-handles.
\label{corr:nicehandlestructure}
\end{corr}

\begin{proof}
If there is annulus bundle, apply Lemma~\ref{lemma:removeannulusbundles} to remove it; if there is a disc bundle whose vertical boundary intersects $\del S_i$, apply Lemma~\ref{lemma:remove_disc_bundles} to remove annulus bundles and disc bundles whose vertical boundaries intersect $\del S_i$.
The resulting handle structure satisfies Convention~\ref{convention:handlestructure} by Lemma~\ref{lemma:removeparallelityhandlestructure}.
The remaining generalised parallelity bundle is a collection of disc bundles with properly-embedded vertical boundary.

Replace these parallelity bundle components with 2-handles and let $\mathcal{H}_i'$ be the resulting collection of handles.
So long as this is a handle structure, it is immediate that it satisfies the definition of being $(0,0)$-nicely embedded subnormal in $\mathcal{H}$.
As any handle that is adjacent to the maximal parallelity bundle $\mathcal{B}$ via a parallelity handle of lower index is itself in $\mathcal{B}$, and the vertical boundary of $\mathcal{B}$ is properly embedded, these new 2-handles $D^2\times D^1$ intersect handles of lower index in all of $\del D^2 \times D^1$.
They are also disjoint from other 2-handles, and any components of their intersection with 1-handles come from components of intersection of 2-handles and 1-handles in $\mathcal{H}_i'$, which is a handle structure, so $\mathcal{H}_i'$ is a handle structure satisfying Convention~\ref{convention:handlestructure}.
\end{proof}

\subsection{Proof of Theorem~\ref{thm:simplicialfibres}}

\begin{prop}
    For each subnormal neighbourhood $S_i$ of a non-multiplicity-two singular fibre with handle structure $\mathcal{H}_i$, we can find a core curve satisfying the conditions in Theorem~\ref{thm:laccorecurves} that is disjoint from the parallelity bundle of $S_i$ with respect to the annulus $A_i$.
    \label{prop:avoidparallelity}
\end{prop}

\begin{proof}
Let $\mathcal{G}_i$ be the handle structure from Corollary~\ref{corr:nicehandlestructure}.
Recall that $\mathcal{G}_i$ has no parallelity handles, is $(0,0)$-nicely embedded subnormal in $M$, its 0- and 1-handles are also 0- and 1-handles of $\mathcal{H}_i$, and $\mathcal{G}_i$ is a solid torus neighbourhood of $\gamma_i$.
A core curve of $\mathcal{G}_i$ thus includes into $\mathcal{H}_i$ as a core curve.
Note that the intersection count conditions in Theorem~\ref{thm:laccorecurves} were satisfied by $\mathcal{H}_i$ as it came from cutting along a normal surface in a handle structure dual to a triangulation.
These conditions are also satisfied in $\mathcal{G}_i$: its 0- and 1-handles are 0- and 1-handles of $\mathcal{H}_i$ respectively, so it suffices to check the intersection conditions on the 2-handles. 
After replacing generalised parallelity bundle components with 2-handles, all intersections between 1-handles and these new 2-handles include into $\mathcal{H}_i$ as components of intersection between the 1-handles and 2-handles.
Thus Theorem~\ref{thm:laccorecurves} applies; that is, there is a core curve $\gamma$ for $S_i$ that, in $\mathcal{G}_i$, has the properties from Theorem~\ref{thm:laccorecurves}, and so in particular is disjoint from the 2-handles in $\mathcal{G}_i$ and hence from the parallelity bundle of $\mathcal{H}_i$.
\end{proof}

\begin{proof}[Proof of Theorem~\ref{thm:simplicialfibres}]
When $M$ is the solid torus, it is already known that there is a core curve of $M$ which is simplicial in $\mathcal{T}^{(51)}$, such that the edges of this simplicial curve are not subsets of edges of $\mathcal{T}^{(50)}$~\cite[Theorem 6.14]{LacSchleimerElliptic}.

Otherwise, recall that $\mathcal{H}$ is dual to the triangulation $\mathcal{T}$ of $M$.
Consider the collection of core curves $\gamma = \{\gamma_i\}$ from Proposition~\ref{prop:avoidparallelity}: this collection consists of one curve isotopic to each non-multiplicity-two singular fibre.
These core curves are contained in the 0- and 1-handles in the handle structures $\mathcal{G}_i$, which are simultaneously $(0,0)$-nicely embedded subnormal in $\mathcal{H}$.
By Theorem~7.7 in~\cite{LacSchleimerElliptic}, the curves $\{\gamma_i\}$ are therefore simultaneously simplicial in $\mathcal{T}^{(79)}$.
Furthermore, the final step of the proof of Theorem~7.7 in~\cite{LacSchleimerElliptic} is to consider some simplicial arcs in a face of $\mathcal{T}^{(77)}$ and to take two barycentric subdivisions to obtain isotopic simplicial arcs that are pushed off this face.
After this step, we see that the pushed off arcs contain only edges of $\mathcal{T}^{(79)}$, and none of $\mathcal{T}^{(78)}$.
In the solid torus case we had this in the setting of $\mathcal{T}^{(51)}$, so we can take 27 further barycentric subdivisions, each time, for each edge of the simplicial core curve, replacing it with two edges which were created in the barycentric subdivision.

Note that, in a barycentric subdivision, the link of an edge $e$ that was created in the subdivision is a circle of four or six edges.
When we take an additional barycentric subdivision, each of these edges is subdivided once, so we have a circle of eight or twelve edges.
After two further barycentric subdivisions, the simplicial neighbourhoods of these simplicial singular fibres are solid tori and do not intersect in their interior; one more suffices to ensure that they are disjoint.
Thus, if we take three barycentric subdivisions to obtain $\mathcal{T}^{(82)}$, we ensure that the simplicial neighbourhoods of the singular fibres are disjoint solid tori and transverse to each edge of each singular fibre there is a meridian curve of length at most $6\times 2^3 = 48$.
\end{proof}

\section{Bounding the triangulation complexity}
\label{section:triangulationbounds}

We will now use the simplicial singular fibres in $\mathcal{T}^{(82)}$ to determine the triangulation complexity of $M$ up to a multiplicative constant.
Both the upper bound and the component of the lower bound that is in terms of the Euler characteristic of the base surface are straightforward.

\begin{figure}[th!]
\centering
  \resizebox{0.25\textwidth}{!}{\includegraphics{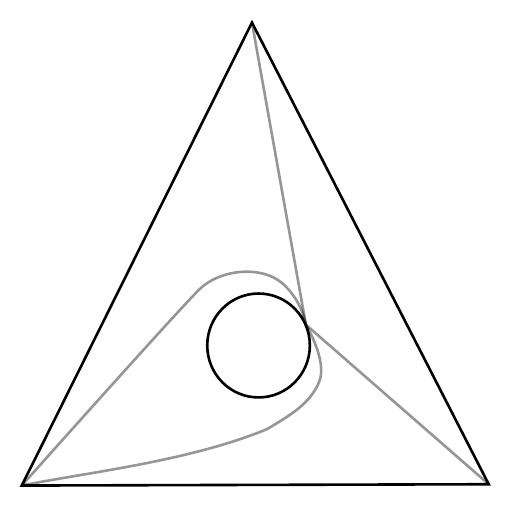}}
  \caption{Subdividing a triangle to add a new boundary component to a surface triangulation. The edges in grey are new internal edges; the central circle is the new boundary component.}
  \label{fig_new_bdry_cpt}
\end{figure}

\begin{prop}
Let $M$ be a manifold with Seifert data $[\Sigma, (p_1, q_1), \ldots, (p_n, q_n)]$ with non-empty boundary.
Then the triangulation complexity is bounded above by
$$\Delta(M) \leq 96|\chi(\Sigma)| + 176 + 70\sum_{i=1}^n  \norm{q_i/p_i}.$$
\label{prop:upper_bound}
\end{prop}

\begin{proof}
Let $b$ be the number of boundary components of $\Sigma$, and let $a$ be, if $\Sigma$ is orientable, twice its genus, and otherwise let it be its nonorientable genus.
Note that $\chi(\Sigma) = 2-a-b$.
Write $\Sigma_n$ for the surface $\Sigma$ with $n$ additional boundary components added.

First, suppose that $\Sigma_n$ is orientable.
When $\Sigma_n$ is an annulus or disc, we can construct a triangulation of size three -- in particular, of size at most $6a-1+4(b+n)$ -- with one vertex on each boundary component.
Otherwise, we can triangulate the closed orientable surface of the correct genus with at most $2a-2$ faces.
We can then add a boundary component in the middle of a face at the cost of 3 additional faces, as shown in Figure~\ref{fig_new_bdry_cpt}, again with one vertex on each boundary component, to give the same bound.
Take the product of this triangulation with an interval to obtain a subdivision of $\Sigma_n\times I$ into triangular prisms, where the same triangulation is induced on the top and bottom faces.
Glue top to bottom to obtain $\Sigma_n\times S^1$, divided into at most $6a-1+4(b+n)$ triangular prisms.

If $\Sigma_n$ is nonorientable, note that there is a triangulation of the M\"obius band using two triangles with two vertices on its boundary, and by adding one more triangle we can ensure that the boundary has a single vertex.
We can add a boundary component to such a surface using three additional triangles, as in the orientable case.
We can use this triangulation of the nonorientable surface $\Gamma_p$ with nonorientable genus one and $p$ boundary components, by a total of $3p$ triangles, to construct a cell structure on $\Gamma_p\ttimes I$ using $3p$ triangular prisms: one over each of these triangles.
Now, we can glue two of these cell structures, for $\Gamma_{p'}\ttimes I$ and $\Gamma_{p''}\times I$, along two annuli $\gamma\times I$ that sit over one each of their boundary components.  
Thus we obtain a cell structure for $\Sigma_n\ttimes I$, and we can glue two of these by the identity map on their boundaries to get a cell structure for $\Sigma_n\ttimes S^1$ composed of at most $2(6a-6+4(b+n))$ triangular prisms.

Pick a subdivision of each quadrilateral face into two triangles, then cone each triangular prism to an interior point, subdividing it into eight tetrahedra: one for each of the faces of this triangulation of the boundary of the prism.
This procedure gives a triangulation of $\Sigma_n\pmttimes S^1$ by at most $16(6a-1+4(b+n))$ tetrahedra.
Note that we were free to choose the triangulation of the quadrilaterals in the boundary.

\begin{figure}[t]
  \centering
  \begin{subfigure}[b]{0.20\textwidth}
  	\centering
   	\resizebox{\textwidth}{!}{\includegraphics{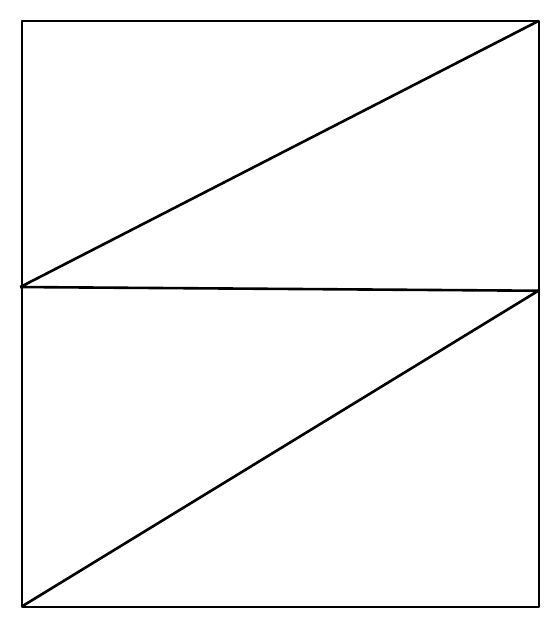}}
	\caption{}
  \end{subfigure}
  \begin{subfigure}[b]{0.20\textwidth}
  	\centering
   	\resizebox{\textwidth}{!}{\includegraphics{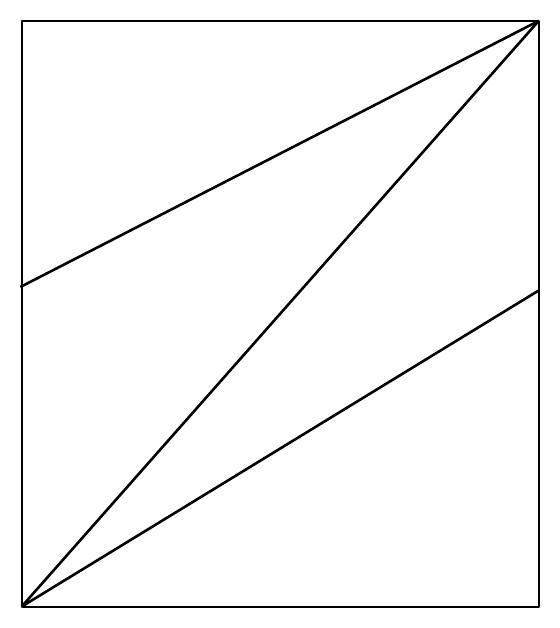}}
	\caption{2-2 move}
  \end{subfigure}
  \begin{subfigure}[b]{0.20\textwidth}
  	\centering
   	\resizebox{\textwidth}{!}{\includegraphics{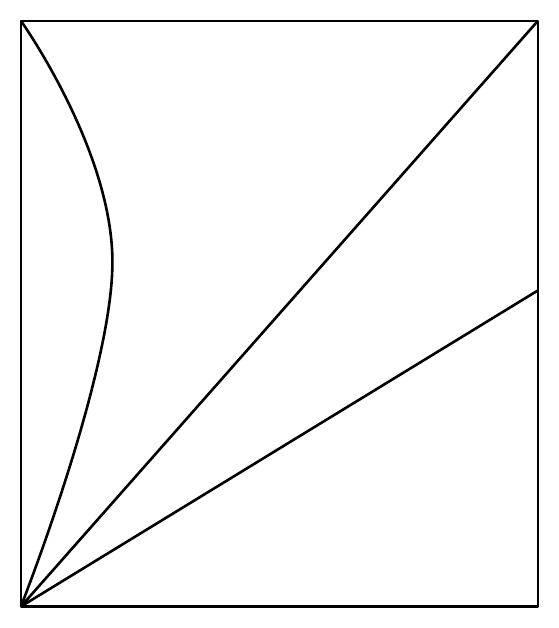}}
	\caption{2-2 move}
  \end{subfigure}
  \begin{subfigure}[b]{0.233\textwidth}
  	\centering
   	\resizebox{\textwidth}{!}{\includegraphics{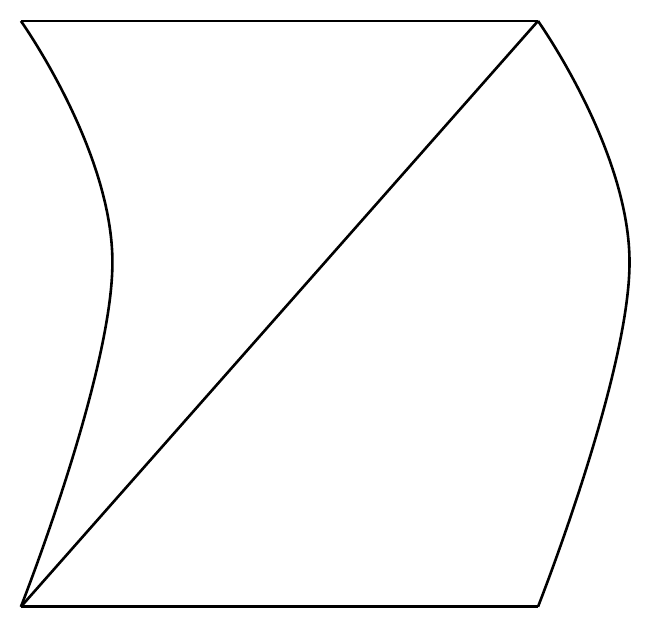}}
	\caption{3-1 move}
  \end{subfigure}
  \caption{Applying three Pachner moves to the triangulation of a boundary torus of $\Sigma_n\times I$ to reduce it to a single vertex triangulation where two of the edges are the fibre and section of the product structure.}
  \label{fig:boundaryreduction}
\end{figure}

Some of these boundary tori correspond to boundary components of $M$, while others will be Dehn filled to create singular fibres.
We restrict our attention to the latter class.
The induced triangulation on each boundary torus depends on the orientability of $\Sigma_n$.
We wish it to be a single vertex triangulation, where the fibre and section curves of the product structure are two of the edges.
If $\Sigma_n$ is orientable, this is automatic.
If it is nonorientable, the initial triangulation of each boundary component is by taking a strip of two rectangles and identifying opposite sides, then subdividing each rectangle into two triangles in a way we can choose.
As shown in Figure~\ref{fig:boundaryreduction}, we can apply three Pachner moves (that is, add three additional tetrahedra) to obtain the required properties.
We have used a total of at most $16(6a-6+4(b+n)) + 3n$ tetrahedra.

Consider a toroidal boundary component that we wish to Dehn fill by gluing the meridian of a solid torus to the slope $q_i/p_i$ curve.
By Proposition~\ref{prop:layeredtriangulationcount}, we can do this using at most $\norm{q_i/p_i}+2$ tetrahedra. 
Thus $\Delta(M)$ satisfies
\begin{align*}
   \Delta(M)    &\leq 16(6a-1+4(b+n)) + 3n + \sum_{i=1}^n (\norm{q_i/p_i}+2)
\end{align*}
so, noting that $\norm{q_i/p_i}$ is at least 1 so $\sum_i \norm{q_i/p_i}$ is at least $n$ and that $|\chi(\Sigma)| = |a + b - 2|$, we have that
\begin{align*}
  \Delta(M)	&\leq 96a - 16 + 64b + 69n + \sum_{i=1}^n \norm{q_i/p_i}\\
  		&\leq |96a + 96b - 192| + 176  + 70\sum_{i=1}^n \norm{q_i/p_i}
\end{align*}
which gives us the bound.
\end{proof}

\begin{lemma}
Let $M$ be a manifold with Seifert data $[\Sigma, (p_1, q_1), \ldots, (p_n, q_n)]$ and non-empty boundary.
Then $\Delta(M)$ is at least $\frac{1}{6}(|\chi(\Sigma)|+1).$
\label{lemma:easylowerbound}
\end{lemma}

\begin{proof}
First, note that as $\Delta(M) \geq 1$, this bound holds when $\Sigma$ is a disc, so we can assume that $\chi(\Sigma)$ is at most 0.
Let $\mathcal{T}$ be a minimal triangulation of $M$.
Note that $\mathcal{T}$ has at most $6|\mathcal{T}|$ edges so its first homology group is of rank at most $6|\mathcal{T}|$.
We can use van Kampen's theorem to construct $\pi_1(M)$, in the same order as in the proof of Proposition~\ref{prop:upper_bound}: first, we take $\Sigma_g^{b+n}\pmttimes S^1$ and then glue in $n$ solid tori with their meridians glued to the specified slopes.
This presentation is given explicitly, for example, in Theorem 2.2.2 of~\cite{brin}.
Abelianising, we find that the first homology has a free summand of rank $|2 -\chi(\Sigma)|$ if $\Sigma$ is orientable, and $|1-\chi(\Sigma)|$ if it is nonorientable.
As $\chi(\Sigma)$ is nonpositive, these are each at least $|\chi(\Sigma)|+1$.
\end{proof}

We now need to show that $\Delta(M)$ is bounded below in terms of the sum of the continued fraction expansions.
The idea is as follows: in $\mathcal{T}^{(82)}$, for each non-multiplicity-two singular fibre we have a simplicial neighbourhood with a meridian curve of length at most 48.
Considering the $(0,0)$-nicely embedded subnormal neighbourhood $\mathcal{G}_i$ of this singular fibre with distinguished annulus $A_i$ in its boundary from Corollary~\ref{corr:nicehandlestructure}, the boundary components of $A_i$ give us a normal curve in the boundary of $\mathcal{G}_i$ of slope $q_i/p_i$ (with respect to some (meridian, longitude) coordinate system on the boundary of the solid torus) which runs through each handle in the boundary at most once.
We thus bring our attention to the complexity of $T^2\times I$ where we have a controlled length curve of a specified slope in each boundary component.
We will need our result to be in terms of handle structures, and so rather than {triangulation} complexity, we give a result about the {tetrahedral} complexity of the handle structure $\mathcal{G}_i$.

The \emph{boundary graph} of a 0-handle $H$ is the cell structure of its intersection with the 1- and 2-handles, as embedded in $\del H \cong S^2$.

\begin{figure}[t]
  \centering
  \begin{subfigure}[b]{0.40\textwidth}
  	\centering
   	\resizebox{0.55\textwidth}{!}{\includegraphics{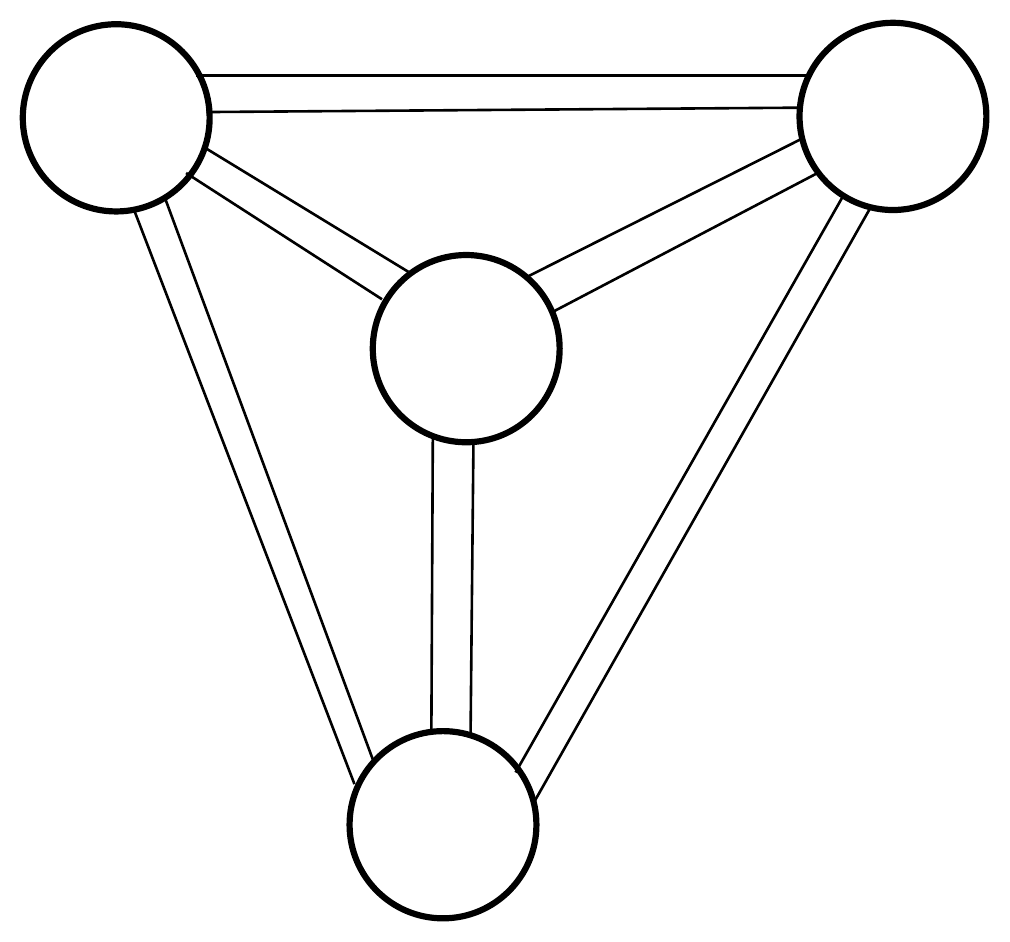}}
	\caption{Tetrahedral boundary graph. Any component of the complement may arise from intersection with $\del M$.}
  \end{subfigure}
  \begin{subfigure}[b]{0.40\textwidth}
  	\centering
   	\resizebox{0.8\textwidth}{!}{\includegraphics{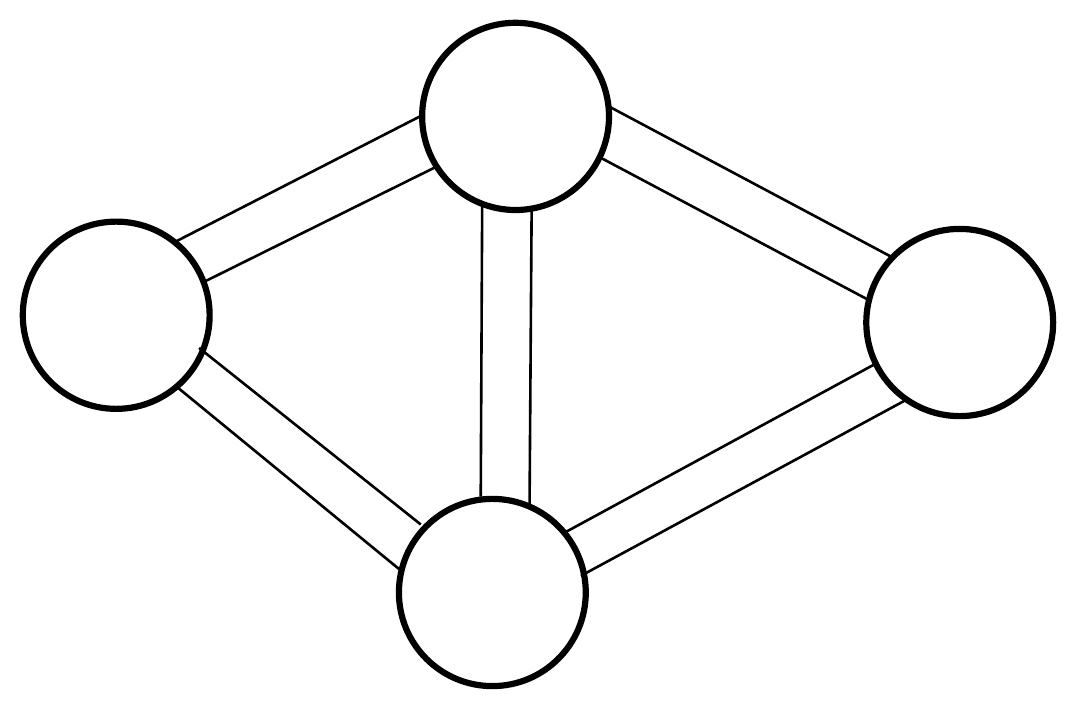}}
	\caption{Semi-tetrahedral boundary graph. One of the components of the complement, which is adjacent to four 1-cells, must arise from $\del M$.}
  \end{subfigure}
  \begin{subfigure}[b]{0.40\textwidth}
  	\centering
   	\resizebox{0.8\textwidth}{!}{\includegraphics{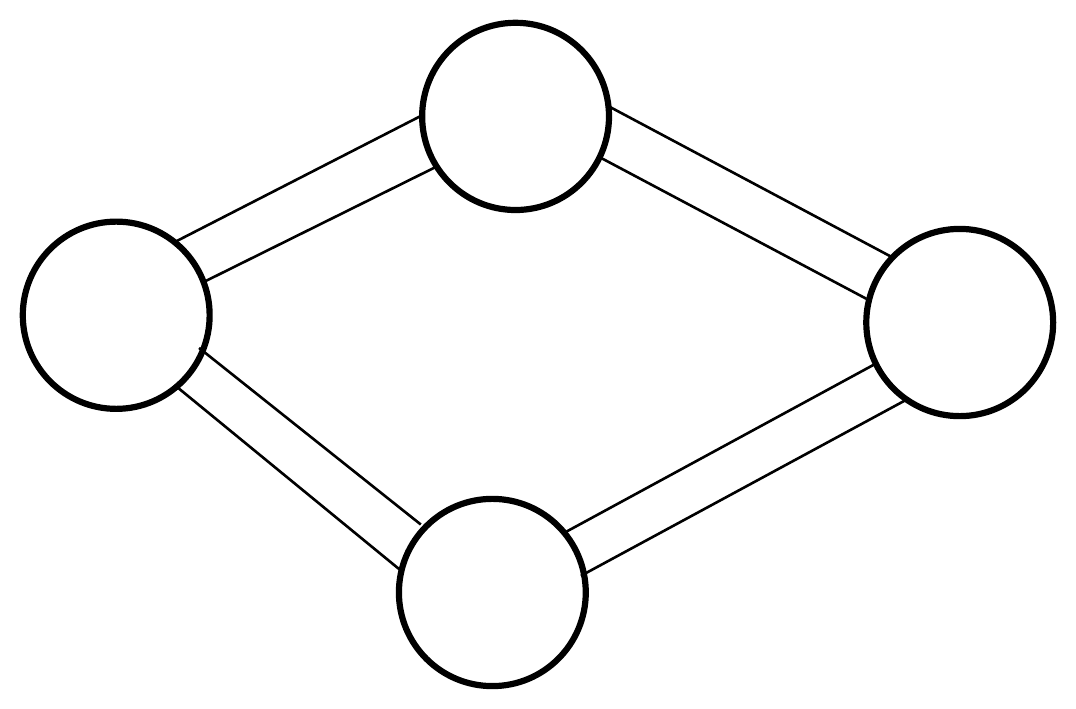}}
	\caption{Parallelity of length four boundary graph. Both of the components of the complement must arise from intersection with $\del M$.}
  \end{subfigure}
  \begin{subfigure}[b]{0.40\textwidth}
  	\centering
   	\resizebox{0.55\textwidth}{!}{\includegraphics{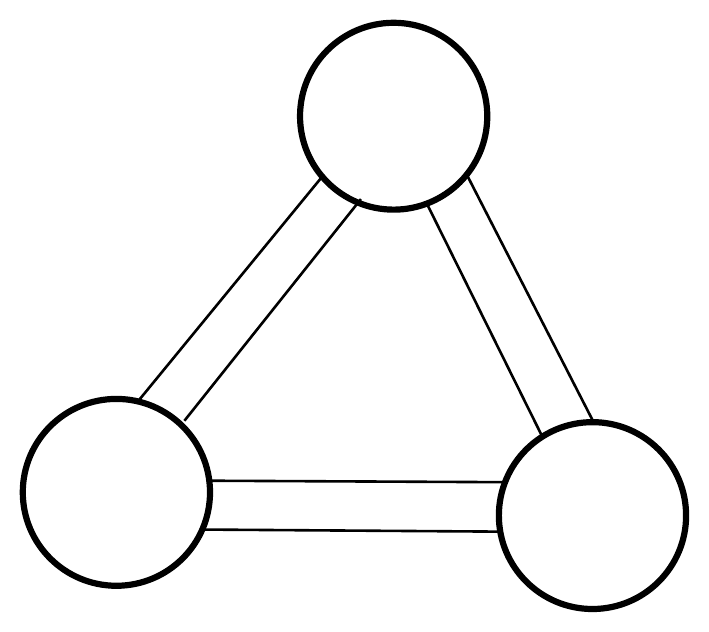}}
	\caption{Product of length three boundary graph. At least one of the components of the complement must arise from intersection with $\del M$.}
  \end{subfigure}
  \caption{The possible boundary graphs in a pre-tetrahedral handle structure.}
  \label{fig:tetrahedralhandlestructure}
\end{figure}

\begin{defn}
A handle structure is \emph{pre-tetrahedral} if the boundary graph of each 0-handle is tetrahedral, semi-tetrahedral, product of length three or parallelity of length four, as illustrated in Figure~\ref{fig:tetrahedralhandlestructure}.

Let $\mathcal{H}$ be a pre-tetrahedral handle structure.
For each 0-handle $H$, let $\alpha$ be the number of components of its intersection with the 3-handles, and let $\beta$ be $\frac{1}{2}$ if $H$ is tetrahedral, $\frac{1}{4}$ if $H$ is semi-tetrahedral, and otherwise zero. 
The \emph{tetrahedral complexity} $\Delta(\mathcal{H})$ of $\mathcal{H}$ is the sum of $\frac{\alpha}{8} + \beta$ over all the 0-handles.
\end{defn}

This definition may seem rather unmotivated; its origin is that if $S$ is a normal surface in a triangulation $\mathcal{T}$ for a closed manifold $M$, and $\mathcal{H}$ is the induced dual handle structure on $M\backslash\backslash S$, then $\Delta(\mathcal{H}) = \Delta(\mathcal{T})$~\cite[Lemma 4.12]{LacPur}.
We will convert our $(0,0)$-nicely embedded subnormal solid tori into pre-tetrahedral handle structures whose tetrahedral complexity is linear in the number of 0-handles they originally contained, and then invoke the following result.

\begin{lemma}
\label{lemma_torus_product_bound}
There is a constant $k > 0$ such that the following holds.
Suppose that $\mathcal{H}$ is a pre-tetrahedral handle structure of $T^2\times [0,1]$ that admits no annular simplification.
Let $\mathcal{C}$ be the cell structure of $T^2\times [0,1]$ where each handle is a 3-cell, and each component of intersection of $j$ handles or of $j-1$ handles and the boundary is a $(4-j)$-cell.
Let $\gamma_0$ and $\gamma_1$ be distinguished essential simple closed curves in the 1-skeleton of $\mathcal{C}$ that are in $T^2\times\{0\}$ and $T^2\times\{1\}$ respectively.
Put coordinates on $T^2\times\{0\}$ such that $\gamma_0$ has slope $1/0$.
Let $q/p$ be the slope of $\gamma_1$ in the coordinates induced from those on $T^2\times\{0\}$ by the product structure.
Suppose that each $\gamma_i$ has length at most $\ell_i$, and that $0 < q < p$.
Then the tetrahedral complexity $\Delta(\mathcal{H})$ is at least $\frac{1}{k}\left(\norm{q/p}-578-96\min(l_0, l_1)\right)$.
\end{lemma}

We first define some terminology.
A \emph{spine} in a closed surface is an embedded graph, with no vertices of degree 1 or 2, whose complement is a disc.
A \emph{cellular spine} in a surface equipped with a cell structure is a spine that is a subset of the 1-skeleton of the cell structure.
We can perform some operations to modify a spine: an \emph{edge contraction} is the operation of collapsing an edge that joins two distinct vertices, while an \emph{edge expansion} is its inverse.
If $\Gamma$ is a spine in a surface $S$, $e$ is an arc properly embedded in the complement of $\Gamma$, and $f$ is an edge of the graph $\Gamma\cup e$ with two different components of $S\backslash\backslash (\Gamma\cup e)$ on each side of it, then an \emph{edge swap} is the operation of removing $f$ from $\Gamma$ and adding $e$.
The \emph{spine graph} of S, denoted $Sp(S)$, is the graph whose vertices are spines of $S$ up to isotopy and whose edges are edge contractions and expansions.

The proof of this lemma is a modification of part of the proof of Theorem 1.1 in~\cite{LacPurSol}.

\begin{proof}
By relabelling we can assume that $\ell_0 \geq \ell_1$.
Consider the cell structure that $\mathcal{C}$ induces on the boundary.
We can use the product structure to compare the cell structures on the two boundary components by projecting them onto the same torus.
Pick a {cellular spine} $\Gamma_0$ for $T^2\times \{0\}$ containing $\gamma_0$.
Such a spine exists as $\gamma_0$ is not separating, so we can take $\Gamma$ to be the union of $\gamma_0$ and some arc in the 1-skeleton that starts and ends on $\gamma_0$.
There is a universal constant $k'$ such that there is a sequence of at most $k'\Delta(\mathcal{H})$ edge contractions and expansions taking $\Gamma_0$ to some cellular spine $\Gamma_1$ in $T^2\times\{1\}$~\cite[Theorem 9.1]{LacPurSol}.
Now there is a sequence of at most $24+4\ell_1$ edge swaps taking $\Gamma_1$ to some cellular spine $\Gamma_2$ for $T^2\times\{1\}$ containing $\gamma_1$~\cite[Lemma 4.15]{LacPurSol}.
As we can perform any edge swap in a spine on a surface $S$ using at most $24 g(S)$ edge expansions and contractions, where $g(S)$ is the genus of $S$~\cite[Lemma 8.3]{LacPur}, we can move from $\Gamma_0$ to $\Gamma_2$ using a total of at most $k'\Delta(\mathcal{H}) + 24(24+4\ell_1)$ edge expansions and contractions.

Consider two complexes associated to the torus: the graph of one-vertex triangulations, $Tr(T^2)$, where the edges are 2-2 Pachner moves, and the spine graph $Sp(S)$.
Each one-vertex triangulation has a spine dual to it, and edges between one-vertex triangulations, which are 2-2-Pachner moves, can be sent to a composition of an edge contraction and an expansion, so there is a quasi-embedding $Tr(T^2)\to Sp(T^2)$ that multiplies all distances by 2.
Note that any vertex of $Sp(T^2)$ is at most distance one from one of these dual spines, so this map is in fact a quasi-isometry.
(In general, for an orientable surface, this argument holds but the distance bound from an arbitrary spine to one dual to a triangulation depends on the genus.)

Now, $\Gamma_0$ and $\Gamma_2$ are within distance one of spines $\Gamma_0'$ and $\Gamma_2'$ with only valence three vertices.
Let $T_0$ and $T_2$ be the dual triangulations to these spines, which contain edges of the same slopes as $\gamma_0$ and $\gamma_1$ respectively.
As we can identify the Farey tree and $Tr(T^2)$, their distance in $Tr(T^2)$ is at least the distance between the lines in the Farey graph whose vertices are the triangulations containing edges of these slopes, which is $\norm{q/p}-1$~\cite[Lemma 5.4]{LacPurSol}.

From the quasi-isometry $\frac{1}{A}d_{Tr(T^2)}(T_0, T_2) - B \leq d_{Sp(T^2)}(\Gamma_0', \Gamma_2')$ for some universal constants $A$ and $B$, so $$\frac{1}{A}(\norm{q/p}-1) - B \leq d_{Sp(T^2)}(\Gamma_0, \Gamma_2) + 2 \leq k'\Delta(\mathcal{H}) + 24(24 + 4\ell_1) + 2$$
and if we pick $k$ to be at least $k'A$ and also at least $2B$, we have the result.
\end{proof}

The constant $k$ in this result relies on the constant in Theorem 6.1~\cite{MMS}, which as far as the author knows has not been bounded.

We will also need the following combinatorial results.

\begin{theorem}[Theorem 7.6, \cite{LacSchleimerElliptic}]
Let $M$ be a compact 3-manifold with a triangulation $\mathcal{T}$. Let $M'$ be a handlebody with a handle structure $\mathcal{H}'$, and let $A$ be a union of disjoint annuli in $\del M'$ that are unions of cells in the induced cell structure on the boundary. Suppose that $M'$ is embedded in $M$ in such a way that $(\mathcal{H}',A)$ is $(0, 0)$-nicely embedded subnormal in the dual of $T$ . Then we can arrange that the following are all simplicial subsets of $T^{(30)}$:
\begin{enumerate}
   \item each zero-handle of $M'$;
   \item each one-handle of $M'$, vertically collapsed onto its co-core;
   \item and the annuli $A$.
\end{enumerate}
\label{thm:embedsimplicial}
\end{theorem}

\begin{lemma}
Let $A$ be a simplicial triangulation of an annulus, composed of $|A|$ triangles.
There is a triangulation of $D^2\times I$, comprising at most $3|A|$ tetrahedra, that induces the triangulation $A$ on the boundary $\del D^2\times I$.
\label{lemma:2disk_construction}
\end{lemma}

\begin{proof}
Consider the cone $CA$ of $A$.
This is homeomorphic to $D^2\times I$ pinched at a point in the centre (the cone point) and has $|A|$ top-dimensional cells.
We have an induced triangulation $X$ of $D^2\times \del I$ from the cone $C\del A$ over $\del A$.
Take the cone over each component of $X$, which is homeomorphic to two balls, and glue it to $CA$ by identifying the inclusions of $C\del A$ into $CX$ and $CA$.
The resulting cell complex $Y$ is homeomorphic to $D^2\times I$ and the triangulation it induces on $\del D^2\times I$ is $A$; as it was built by gluing together cones over triangulated surfaces, each 3-cell is a tetrahedron.
We note that $Y$ contains $|A|+|\del A|$ top-dimensional cells, where $|\del A|$ is the length of the boundary of $A$, and that $|\del A|$ is at most $2|A|$ since no triangle of $A$ has all three edges in the boundary.
\end{proof}

\begin{lemma}
Let $\mathcal{T}$ be a material triangulation of a compact 3-manifold $M$.
There is a pre-tetrahedral handle structure $\mathcal{H}$ for $M$, containing at most $5|\mathcal{T}|$ 0-handles, that does not contain any parallelity handles with respect to $\del M$, and such that the induced cell structure on the boundary is the same as the boundary triangulation induced by $\mathcal{T}$.
\label{lemma:inducedhandlestructurenoparallelity}
\end{lemma}

\begin{proof}
The idea here is to take the dual handle structure to $\mathcal{T}$ and rule out the possibility of any parallelity handles by ensuring that each handle has at most one component of intersection with the boundary.
Note that a (2-dimensional) cell structure $A$ on a closed surface $S$ induces a (3-dimensional) cell structure on $S\times I$, which we denote as $A\times I$, whose induced boundary cell structure on each boundary component is a copy of $A$, and where the intersection of each cell with a boundary component is connected.
We can consider our triangulation $\mathcal{T}$ to be a cell structure.
Let $A$ be the induced cell structure on $\del M$ from $\mathcal{T}$.
Let $\mathcal{C}$ be constructed from the cell structure $\mathcal{T}$ by attaching a copy of $A\times I$ to $\del\mathcal{T}$.
Let $\mathcal{H}$ be the dual handle structure to $\mathcal{C}$, as in Definition~\ref{defn:dual_handle}.
The only handles that intersect the boundary are those added by attaching $A\times I$, each of whose cells intersect each boundary component of $A\times I$ in a connected subsurface, and hence intersect $\del M$ in at most one component.
As the triangulation of $\del M$ from $\mathcal{T}$ contained at most $4|\mathcal{T}|$ faces, $\mathcal{C}$ has at most $5|\mathcal{T}|$ 3-cells, so $\mathcal{H}$ has at most $5|\mathcal{T}|$ 0-handles.

Now, $\mathcal{H}$ is pre-tetrahedral.
The interior 0-handles are dual to the interior cells of $\mathcal{C}$ which are tetrahedra, so these 0-handles have tetrahedral boundary graphs.
Those that intersect the boundary are semi-tetrahedral: they are dual to cells that are the product of a triangle and an interval, where one of the triangular faces is on the boundary.
\end{proof}

We can now prove the bound in the absence of multiplicity two singular fibres.

\begin{prop}
Let $M$ be a manifold with Seifert data $[\Sigma, (p_1, q_1), \ldots, (p_n, q_n)]$ other than the solid torus where the first $m$ singular fibres in this list are not of multiplicity two.
Then $\Delta(M) \geq \frac{1}{6}m$, and there is a universal constant $k>0$ such that
$$\Delta(M) \geq \frac{1}{k}\left(\sum_{i=1}^m \norm{q_i/p_i}-9794m\right).$$
\label{prop:no_mult_2}
\end{prop}

\begin{proof}
Let $\mathcal{T}$ be a triangulation of $M$, and let $\mathcal{H}$ be the dual handle structure.
Consider, for each singular fibre $\gamma_i$ that is not of multiplicity two, the $(0,0)$-nicely embedded subnormal handle structure $\mathcal{G}_i$ in $M$ for the subnormal neighbourhood $(S_i, A_i)$ of $\gamma$ from Corollary~\ref{corr:nicehandlestructure}, which exists as $M$ is not a solid torus.
Let $\mathcal{C}$ be the disjoint union of these handle structures.

Each subnormal neighbourhood $\mathcal{G}_i$ contains at least one non-parallelity subnormal 0-handle in $\mathcal{H}$, and each 0-handle of $\mathcal{H}$ contains at most six such possible subnormal 0-handles.
Thus, as each 0-handle in $\mathcal{H}$ is dual to a tetrahedron in $\mathcal{T}$, $m\leq 6|\mathcal{T}|$.

As the subnormal neighbourhoods are simultaneously $(0,0)$-nicely embedded, by Theorem~\ref{thm:embedsimplicial} in $\mathcal{T}^{(30)}$ we can arrange that the following parts of $\mathcal{C}$ are simplicial: the 0-handles, the cocores of the 1-handles, and the attaching annuli of the 2-handles.
Note that we already know from Theorem~\ref{thm:simplicialfibres} that after an additions 82 barycentric subdivisions, a core curve of each of these solid tori is simplicial and that these simplicial curves have disjoint solid tori as their simplicial neighbourhoods, with a simplicial meridian curve of length 48 for each such neighbourhood.

Subdivide one more time, giving $\mathcal{T}^{(113)}$, so that the simplicial neighbourhood of each core curve is disjoint from the attaching annuli of the 2-handles of $\mathcal{C}$.
Doing this doubles the degree of each edge, so increases the length of the simplicial meridian curve to 96, and ensures that the simplicial neighbourhoods are disjoint from $\del \mathcal{G}_i$ for each $i$.
The induced triangulations of the attaching annuli in $\mathcal{T}^{(113)}$ consist of subdivided triangles from $\mathcal{T}^{(30)}$.
Each of these faces is subdivided into six triangles in each barycentric subdivision, so the attaching annuli triangulations contain a total of at most $6^{83}\times 4\times 24^{30}|\mathcal{T}|$ triangles.
By Lemma~\ref{lemma:2disk_construction} we can triangulate the 2-handles (although not necessarily as a subset of $\mathcal{T}^{(113)}$) using at most $3\times 6^{83}\times 4\times 24^{30}|\mathcal{T}|$ additional tetrahedra, so we have a total of at most $2\times 24^{113}|\mathcal{T}|$ tetrahedra.
Now, restrict to one of these triangulations of a subnormal neighbourhood $(S_i, A_i)$ of a core curve, which is a triangulation of the handle structure $\mathcal{G}_i$. 
The annulus $A_i$ is simplicial, as it was a subset of the cell structure of the boundary, and so its boundary curves are simplicial.
Take one of these.
This curve has slope $(q_i, p_i)$, the relevant Seifert data for that singular fibre for some choice of (meridian, longitude) coordinates on the boundary of the solid torus.

Drill out a simplicial neighbourhood of each of the core curves, which we recall is disjoint by construction from the attaching annuli of the 2-handles and from $\del \mathcal{G}_i = \del S_i$ and has a meridian of length at most 96.
We then have a copy of $T^2\times I$ for each singular fibre, triangulated using a total of at most $2\times 24^{113}|\mathcal{T}|$ tetrahedra, each with coordinates such that they have a $(1,0)$ curve of length at most 96 on one side and, on the other, a simplicial $(q_i, p_i)$ curve.
By Lemma~\ref{lemma:inducedhandlestructurenoparallelity} there is a pre-tetrahedral handle structure for this collection of copies of $T^2\times I$, inducing the same triangulation of the boundary, that has no parallelity handles and contains at most $10\times 24^{113}|\mathcal{T}|$ 0-handles.
Each 0-handle contributes at most $1$ to the tetrahedral complexity of the handle structure so the tetrahedral complexity of the collection is at most $10\times 24^{113}|\mathcal{T}|$.
By Lemma~\ref{lemma_torus_product_bound} we thus have
$$|\mathcal{T}| \geq \frac{1}{10\cdot 24^{113}k'}\left(\sum_{i=1}^m \norm{q_i/p_i}-9794m\right)$$
for some $k' > 0$.
\end{proof}

\begin{corr}
Let $M$ be a manifold with Seifert data $[\Sigma_g^b, (p_1, q_1), \ldots, (p_n, q_n)]$ where the first $m$ singular fibres in this list are not of multiplicity two.
Then there is a universal constant $k$ such that
$$\Delta(M) \geq \left(\frac{1}{24}-\frac{9794}{k}\right)m + \frac{1}{24}|\chi(\Sigma)| + \frac{1}{k} \sum_{i=1}^m \norm{q_i/p_i} + \frac{7}{24}.$$
\label{corr:lowerboundnomult2}
\end{corr}

\begin{proof}
Let $\mathcal{T}$ be a triangulation of $M$.
From Lemma~\ref{lemma:easylowerbound} and Proposition~\ref{prop:no_mult_2} we have four different lower bounds for $|\mathcal{T}|$ including the trivial one: that it is at least one.
The maximum of these is certainly at least the average of the four, so as we can set $k$ to also be at least one,
\begin{align*}
    |\mathcal{T}| &\geq \frac{1}{4}\left(\frac{m}{6} + \frac{1}{6}(|\chi(\Sigma)| + 1) + \frac{1}{k}\sum_{i=1}^m \norm{q_i/p_i} - \frac{9794}{k}m + 1\right)
\end{align*}
which gives the bound.
\end{proof}

\begin{proof}[Proof of Theorem~\ref{thm:complexitybound}]
The following upper bound is given in Proposition~\ref{prop:upper_bound}: 
$$\Delta(M)\leq 96|\chi(\Sigma)| + 176 + 70\sum_{i=1}^n \norm{q_i/p_i}.$$

The lower bound when there are no singular fibres of multiplicity two follows from Proposition~\ref{prop:no_mult_2}.
To extend our bound to multiplicity two singular fibres, we take a double cover to remove all but at most one of them.

Consider the degree two cover of $M$ constructed as follows: view $M$ as a circle bundle over the orbifold $S$ that is $\Sigma$ with $n$ cone points.
The multiplicity two singular fibres are the preimages of degree two cone points.
Pair these cone points (with at most one left over) by picking some disjoint family of arcs that begin and end on degree two points.
Let $s$ be the number of arcs.
There is a homomorphism $\pi_1(S)\to \ZZ/2$ by taking a loop to its intersection count with this family of arcs mod 2.
The degree two cover $\widetilde{S}\to S$ associated to the index 2 kernel of this homomorphism induces a horizontal degree two cover $\widetilde{M} \to M$ on the level of 3-manifolds.
(Equivalently, we can consider the collection of Klein bottles in $M$ that is the preimage of the arcs in $S$ and take the corresponding degree two cover.)

Which Seifert fibred manifold is $\widetilde{M}$?
Using the Riemann-Hurwitz formula, the base surface $\widetilde{S}$ of $\widetilde{M}$ has Euler characteristic $2\chi(S) -\sum_{i=1}^{2s} 1 = 4-2a-2b-2s$, where $a$ is the nonorientable genus or twice the orientable genus, according to whether $S$ is nonorientable or not, and $b$ is the number of boundary components of $S$.
It has $2b$ boundary components, two singular fibres of index $\pm q_i/p_i$ for each non-multiplicity-two singular fibre of $\widetilde{M}$, as well as, if there was a multiplicity two singular fibre left over, two multiplicity two singular fibres.
A $-q_i/p_i$ singular fibre, which is a $(p_i-q_i)/p_i$ singular fibre after a change of coordinates, contributes the same continued fraction sum by Lemma~\ref{lemma:negative_same_continued_sum}.
We deduce that if $\widetilde{S}$ is orientable, it has genus $a + s-1$, or if it is nonorientable, it has $2a + 2s-2$ crosscaps.
Sort the list of singular fibres of $M$ was so that those with multiplicity two are at the end, and let $m$ be the number of non-multiplicity-two singular fibres.

When $a + b - 2 \geq 0$, $|s+a+b-2| = |s| + |a+b-2|$.
Otherwise, in which case $\Sigma$ is a disc, $|s + a + b-2| = |s|-1$ so either way $|s+a+b-2| +2 \geq |s| + |a+b-2|$.
We can take $k$ to be large enough that $\frac{1}{24}-\frac{9794}{k}$ is at least $\frac{1}{48}$.
Using these two bounds, we can modify the lower bound on $\Delta(\widetilde{M})$ from Corollary~\ref{corr:lowerboundnomult2} in the following way:
\begin{align*}
\Delta(\widetilde{M}) &\geq \left(\frac{1}{24}-\frac{9794}{k}\right)2m + \frac{1}{24}|4-2a-2b-2s| + \frac{1}{24}\cdot 4 + \frac{1}{k} \sum_{i=1}^m2\norm{q_i/p_i} + \frac{3}{24}\\
&\geq 2\cdot\frac{1}{48}m + \frac{1}{12}(|s| + |a + b-2|) + \frac{2}{k}\sum_{i=1}^m\norm{q_i/p_i} + \frac{3}{24}\\
&\geq \frac{1}{24}m + \frac{1}{12}|\chi(S)| + \frac{1}{12}s + \frac{2}{k}\sum_{i=1}^m\norm{q_i/p_i} + \frac{3}{24}
\end{align*}
but then each multiplicity two fibre contributes 2 to the sum of $\norm{q_i/p_i}$ over all the singular fibres, so $\sum_{i=1}^n \norm{q_i/p_i} \leq \sum_{i=1}^m \norm{q_i/p_i} + 4s + 2$.
By taking $k$ to be at least 96, we have that $\frac{4}{k}$ is at most $\frac{1}{24}$ and $\frac{8}{k}s$ is at most $\frac{1}{12}s$.
We can thus obtain the bound that
$$\Delta(\widetilde{M}) \geq \frac{1}{24}m + \frac{1}{12}|\chi(S)| + \frac{2}{k}\sum_{i=1}^n\norm{q_i/p_i} + \frac{1}{12}.$$
By lifting a minimal triangulation of $M$ through the double cover, we see that $\Delta(\widetilde{M}) \leq 2\Delta(M)$.
Thus we have the result for a suitable choice of $k$.
\end{proof}

\section{Further questions}
There are several lingering questions from this work.

We showed in Theorem~\ref{thm:simplicialfibres} that the non-multiplicity two singular fibres are always simplicial in the $79^{th}$ barycentric subdivision of a triangulation of a Seifert fibered space with non-empty boundary.
Is this true for the multiplicity two ones also? The barrier in our approach to this case was that one cannot rule out parallelity bundle components that are thickened M\"obius bands.
Similarly, is this true for closed Seifert fibered spaces as well?
We made great use of the fact that the annuli cutting off the singular fibres are essential.
In the closed case, one must handle not only that the tori cutting out singular fibre neighbourhoods are trivial, but also that in some small Seifert fibered spaces there are no incompressible orientable surfaces at all.

Can one extend the bounds on triangulation complexity in Theorem~\ref{thm:complexitybound} to closed Seifert fibered spaces as well?
Let $M$ be such a manifold with Seifert data $[\Sigma, e, (p_1,q_1),\ldots,(p_n,q_n)]$, where $e$ is the Euler number and $\Sigma$ is a closed surface of genus $g$.
If $M$ has an essential vertical torus then we have a vertex normal essential torus by Corollary 6.8 of~\cite{JacoTollefson}, so outside of a few cases in which there is a horizontal torus we find that there must be a fundamental essential vertical torus.
Cutting along this torus, we get a handle structure for the resulting manifold whose size is bounded above and below by a linear function of the size of the original triangulation~\cite[Theorem 9.3]{LackenbyCertificationKnottedness}, and so we can see that $\Delta(M)$ is at least $k'(g+\sum_{i=1}^n \norm{q_i/p_i})$ for some $k'$.
But this lacks any reliance on Euler number, so cannot be an upper bound up to multiplicative constants.
We know by a similar proof to Proposition~\ref{prop:upper_bound} that $\Delta(M) \leq k(g+b+|e|+\sum_{i=1}^n \norm{q_i/p_i})$ for any closed Seifert fibered space.
We conjecture that (up to a multiplicative constant) this is a lower bound as well.

Can one make the constants in Theorem~\ref{thm:complexitybound} explicit? The sole impediment in this method of proof is that the constant in Lemma~\ref{lemma_torus_product_bound} relies on the nonconstructive constant in Theorem 6.1 of Masur, Mosher and Schleimer's work~\cite{MMS}.

One motivation for considering the triangulation complexity of Seifert fibered spaces is that the lower bound by the simplicial volume of the manifold tells us nothing, for their simplicial volumes are zero.
The other class of irreducible 3-manifolds with zero Gromov norm is the graph manifolds.
It remains to determine their triangulation complexities (up to multiplicative constants) in terms of their topology.
Let $\mathcal{T}$ be a triangulation of such a manifold $M$.
Let $\{T_i\}$ be a maximal collection of disjoint normal tori giving the JSJ decomposition, so cutting $M$ into a collection of Seifert fibered spaces $\{M_j\}$.
As these tori are essential, the generalised parallelity bundle of the collection $(\{M_j\}, \{T_i\})$ is a collection of annulus, M\"obius band and disc bundles, and using the algorithm of Agol, Hass and Thurston~\cite{AgolHassThurston}, as applied by Lackenby~\cite[Theorem 9.3]{LackenbyCertificationKnottedness}, we can compute a triangulation of $(\{M_j\})$ whose size is linear in the size of $\mathcal{T}$.
This strategy tells us that $\Delta(M) \geq k\sum_j \Delta(M_j)$ for some $k$.
However, this cannot be optimal as there are in general an infinite number of graph manifolds built from a given set of Seifert fibered blocks; in this computation, much as in the closed Seifert fibered space situation, we entirely lose the information in the gluing maps.

\bibliographystyle{abbrv}
\bibliography{bibliography}

\end{document}